\documentclass[a4paper,11pt]{amsart}
\usepackage[utf8]{inputenc}
\usepackage{amsmath,amssymb,amsfonts,mathrsfs,mathtools}
\usepackage{graphicx}
\usepackage[colorlinks=true, linkcolor=red, citecolor=blue]{hyperref}
\usepackage{tikz}
\numberwithin{equation}{section}

\newtheorem{theorem}{Theorem}[section]
\newtheorem{proposition}[theorem]{Proposition}
\newtheorem{lemma}[theorem]{Lemma}

\theoremstyle{remark}
\newtheorem{remark}{Remark}[section]

\newcommand{\dis}{\displaystyle}

\newcommand{\vertiii}[1]{\left\vert\!\left\vert\!\left\vert #1 \right\vert\!\right\vert\!\right\vert}  
\numberwithin{equation}{section}
\makeatletter
\@namedef{subjclassname@2010}{\textup{2020} Mathematics Subject Classification}
\makeatother

\begin{document}
	
	\pagenumbering{arabic}
	
	\title[High order KP system: Control and stabilization results]{Unique Continuation for Fifth-Order KP Equation and its application to control problems}
		
	
	\author[Capistrano-Filho]{Roberto de A. Capistrano-Filho\textsuperscript{*}}
	\address{Departamento de Matem\'atica, Universidade Federal de Pernambuco (UFPE), 50740-545, Recife (PE), Brazil}
	\email{\url{roberto.capistranofilho@ufpe.br}}
	
	\author[Nascimento]{Ailton C. Nascimento}
	\address{Departamento de Matem\'atica, Universidade Federal do Piau\'i (UFPI), 64049-550, Teresina (PI), Brazil}
	\email{\url{ailton.nascimento@ufpi.edu.br}}
	
	\thanks{\textsuperscript{*}Corresponding author: \url{roberto.capistranofilho@ufpe.br}}
	

\begin{abstract}
We develop a framework for the fifth-order Kadomtsev--Petviashvili equation on $\mathbb{T}_x \times \mathbb{R}_y$ within a mean-zero KP-adapted Sobolev scale. A localized high-order feedback acting on the periodic variable yields a $5/2$--derivative gain in suitable space--time norms, leading to propagation of regularity and a unique continuation property for the linear dynamics. As a consequence, we derive an observability inequality for the adjoint system and establish exponential stabilization of the nonlinear closed-loop equation: for small initial data in $X_{s,0}$, $s>2$, solutions are global and decay exponentially in $X_s$. Combining observability with the Hilbert Uniqueness Method and a fixed-point argument, we obtain local exact controllability near the origin, with $L^2$ controls supported in the feedback region and cost linear in the data size. The analysis relies on a novel combination of unique continuation, frequency grouping, and the one-sided Fourier vanishing mechanism introduced for the Benjamin--Ono equation by Linares and Rosier in \textit{Trans. Amer. Math. Soc.} (2015)~\cite{LR}, here extended to the fifth-order Kadomtsev--Petviashvili equation.
\end{abstract}
	
	\maketitle
	
	\section{Introduction}
We study the fifth-order Kadomtsev-Petviashvili (KP5) equation posed on the cylindrical domain $\Omega:=\mathbb{T}_x\times\mathbb{R}_y$, namely
\begin{equation}\label{aKPe2D}
\left\{
\begin{aligned}
\partial_t u + \beta\,\partial_x^{5}u + u\,\partial_x u
+ \gamma\,\partial_x^{-1}\partial_y^{2}u &= 0,\\
u(x,y,0) &= \phi(x,y),
\end{aligned}
\right.
\end{equation}
where $\beta\in\mathbb{R}\setminus\{0\}$ and $\gamma\in\{\pm1\}$.
The nonlocal operator $\partial_x^{-1}$ is defined on functions with zero
mean in the periodic variable via
\begin{equation}\label{antDer}
\widehat{(\partial_x^{-1}f)}(k,\eta)
= \frac{1}{ik}\,\widehat{f}(k,\eta),
\qquad (k,\eta)\in(\mathbb{Z}\setminus\{0\})\times\mathbb{R},
\end{equation}
and throughout the paper, we impose the constraint
\begin{equation}\label{eq:mean-zero-prelim}
\int_{\mathbb{T}} u(x,y,t)\,dx = 0,
\qquad \text{for a.e. }(y,t)\in\mathbb{R}\times\mathbb{R}_+,
\end{equation}
which ensures that $\partial_x^{-1}$ is well defined via \eqref{antDer}, and, now on, will be said $x$-mean-zero condition.

\subsection{Background and motivation} The Kadomtsev-Petviashvili (KP) equation was originally introduced to describe the transverse stability of KdV solitary waves \cite{KP1970}. In its classical third-order form, it has been extensively analyzed in
various geometries, including $\mathbb{R}^2$, mixed domains such as $\mathbb{R}\times\mathbb{T}$, and fully periodic settings; see, e.g., \cite{Molinet2011,Tzvetkov2008,IoKe2007}.

Higher-order KP-type models arise naturally when refining long-wave asymptotics beyond the classical regime. The fifth-order KP equation \eqref{aKPe2D} incorporates higher-order longitudinal dispersion through $\partial_x^5$, while retaining the characteristic transverse coupling $\partial_x^{-1}\partial_y^2$. Such equations appear in water wave theory, plasma physics, and nonlinear lattice dynamics; see \cite{Craig2005,Erbay2022}. Well-posedness results for fifth-order KP equations on $\mathbb{R}^2$ have been obtained in various anisotropic Sobolev frameworks; see \cite{LiXiao2008,LiYanZhang2017,Boukarou2021}.

\subsection{Analytical framework and stabilization mechanism} The main objective of this work is to establish a unique continuation property, which plays a central role in the derivation of the control results. More precisely, we aim to design a stabilization mechanism for \eqref{aKPe2D} based on a localized high-order feedback acting in the periodic direction. 

Motivated by the one-dimensional theory for higher-order KdV equations, we introduce the following feedback operator
$$
F(u):=-G D_x^{5} G u,
$$
where $G$ is a localized averaging operator in the $x$-variable, and $D_x^r$ denotes the homogeneous Fourier multiplier with symbol $|k|^r$ for $k\neq0$.  A key feature of this feedback is the dissipation mechanism it induces. At the linear level, one obtains the energy identity
\[
\frac{1}{2}\|v(T)\|_{L^2}^2
+\int_0^T \|D_x^{5/2}(G v)\|_{L^2}^2\,dt
=\frac{1}{2}\|v_0\|_{L^2}^2,
\]
which reveals a localized gain of $5/2$ derivatives in the control region. This smoothing effect can be propagated to the whole domain through a propagation-of-regularity argument, yielding a global $5/2$-order gain in appropriate functional spaces.

Fix \(g\in C^\infty(\mathbb{T})\), \(g\ge0\), \(g\not\equiv 0\), normalized by
$$
	\int_{\mathbb{T}} g(x)\,dx = 1,
$$
and set
\[
\omega:=\{x\in\mathbb{T}:\ g(x)>0\}.
\]
We define a bounded self-adjoint operator \(G\) acting only in \(x\) (fiberwise in \(y\)) by
\begin{equation}\label{eq:G-def-prelim}
	(Gf)(x,y)
	:= g(x)\left(f(x,y) - \int_{\mathbb{T}} g(\tilde x)\,f(\tilde x,y)\,d\tilde x\right).
\end{equation}
The closed-loop term \(G D_x^5 G u\) is therefore supported in \(\omega\) and preserves the \(x\)-mean-zero structure.

The nonlinear feedback system we study is
\begin{equation}\label{eq:nonlin-closed-loop}
\begin{cases}
	\partial_t u
	+ \beta\,\partial_x^5 u
	+ u\,\partial_x u
	+ \gamma\,\partial_x^{-1}\partial_y^{2}u
	+ G D_x^{5}G u
	= 0,\\
u(0)=u_0.
\end{cases}
\end{equation}
To capture this structure, we work in the KP-adapted Sobolev scale
$$
X_s(\Omega)
:= \left\{u\in H^s(\Omega):\ \partial_x^{-1}u\in H^s(\Omega)\right\},
$$
endowed with the natural norm
\[
\|u\|_{X_s}^2 = \|u\|_{H^s}^2 + \|\partial_x^{-1}u\|_{H^s}^2,
\]
and its mean-zero subspace $X_{s,0}(\Omega)$. Within this framework, we first establish a unique continuation property for the linearized KP5 equation with localized damping. This result, as mentioned before, plays a central role, as it leads to an observability inequality and, consequently, to exponential decay for the associated linear semigroup.

\subsection{Main results} Our first main result is a unique continuation principle for the KP5 flow, which ensures that solutions vanishing on the control region must vanish identically. The result can be read as follows. 

\begin{theorem}\label{Th:UCP-5KP}
	Let $T>0$, $c\in L^2(0,T)$, and $v(x,t)\in L^2\!\left((0,T);L^2(\mathbb{T}\times\mathbb{R})\right)$ be a distributional solution of
	\begin{equation}\label{eq:UCP-assumption}
		\left\{
		\begin{aligned}
			\partial_t v
			+\beta\,\partial_x^5 v
			+\gamma\,\partial_x^{-1}\partial_y^2 v
			+\varepsilon D_x^5 v &=0
			&& \text{in }\mathbb{T}\times\mathbb{R}\times(0,T),\\
			v(x,y,t)&=c(t)
			&& \text{for a.e. }(x,y,t)\in\omega\times\mathbb{R}\times(0,T),
		\end{aligned}
		\right.
	\end{equation}
	satisfying the $x$-mean-zero condition \eqref{eq:mean-zero-prelim}.
	Then $v\equiv 0$ a.e.\ in $\mathbb{T}\times\mathbb{R}\times(0,T)$.
\end{theorem}

Building on this property and on the smoothing effect induced by the feedback, we prove that the nonlinear closed-loop system \eqref{eq:nonlin-closed-loop} is globally well-posed for small initial data in $X_{s,0}(\Omega)$, $s>2$, and that its solutions decay exponentially in time. 

\begin{theorem}\label{thm:stab5KP}
	Let \(s>2\). There exist constants \(\rho,\lambda,C>0\) such that for any\(u_0\in X_{s,0}(\Omega)\) with $\|u_0\|_{X_s}\le \rho$,  the problem \eqref{eq:nonlin-closed-loop} admits a unique global solution
	\[
	u\in C([0,\infty);X_{s,0}(\Omega))
	\cap L^2_{\mathrm{loc}}([0,\infty);X_{s+5/2,0}(\Omega)).
	\]
	Moreover,
	\begin{equation}\label{eq:nonlin-exp-decay-kp5}
		\|u(t)\|_{X_s}
		\le C e^{-\lambda t}\|u_0\|_{X_s},
		\qquad\forall\,t\ge0,
	\end{equation}
	or equivalently,  \(u\) decays exponentially in \(X_s\).
\end{theorem}

Finally, we derive a local exact controllability result for the controlled system where the control is supported in the same region as the feedback. 

\begin{theorem}\label{thm:control5KP}
	Let \(T>0\) and \(s>2\). There exists \(\delta=\delta(s,T)>0\) such that, for any
	\(u_0,u_1\in X_{s,0}(\Omega)\) with $\|u_0\|_{X_s}+\|u_1\|_{X_s}\le \delta$, there exists $q\in L^2\left(0,T;X_{s,0}(\Omega)\right)$, such that the controlled system
	\begin{equation}\label{control5KP_q}
		\left\{
		\begin{aligned}
			\partial_t u
			+ \beta\,\partial_x^5 u
			+ \gamma\,\partial_x^{-1}\partial_y^2 u
			+ u\,\partial_x u
			+ G D_x^{5}G u
			&= (D_x^{5/2}G)^\ast q,
			&& (x,y,t)\in\Omega\times(0,T),\\
			u(0) &= u_0,\quad u(T)=u_1,
		\end{aligned}
		\right.
	\end{equation}
	admits a solution
	\[
	u\in C([0,T];X_{s,0}(\Omega))
	\cap L^2\left(0,T;X_{s+5/2,0}(\Omega)\right).
	\]
	Moreover,
	\[
	\|q\|_{L^2(0,T;X_s)}
	\le C(s,T)\left(\|u_0\|_{X_s}+\|u_1\|_{X_s}\right),
	\]
	for some constant \(C(s,T)>0\).
\end{theorem}

The result is obtained by combining the Hilbert Uniqueness Method, in the framework of Lions \cite{Lions1988}, with a perturbative argument. More precisely, we prove that any sufficiently small initial and target states in $X_{s,0}$ can be connected in finite time by means of an $L^2$-control, with a control cost that depends linearly on the size of the data.

\begin{remark} Let us give some remarks.

	\begin{itemize} 
	\item[a.] For the quadratic nonlinearity \(u\,\partial_x u\), the linear part of the analysis, in particular, the unique continuation property  (Theorem \ref{Th:UCP-5KP}), the observability inequality, and the exponential decay of the linear semigroup, extends to every \(s\ge 0\). By contrast, the restriction \(s>2\) enters only in the nonlinear theory, through the bilinear estimate of Lemma \ref{lem:bilinear-kp5-X}. More precisely, the argument used there relies on the product structure in the KP-adapted Sobolev scale at the level \(s-5/2\), which requires \(s-5/2>-1/2\), that is, \(s>2\).
		\item[b.] Accordingly, the threshold \(s>2\) is not imposed by the linear stabilization mechanism itself, but by the nonlinear estimates needed to close the fixed-point argument for the term \(u\,\partial_x u\). Whether the nonlinear stabilization and local exact controllability results remain valid for \(s\in[0,2]\) is an interesting open question, which would require bilinear estimates sharper than those provided by the tame product framework used here.
		\end{itemize}
\end{remark}

\subsection{Heuristic and novelties of the work}  In this work, our analysis combines semigroup methods, propagation of regularity, and nonlinear fixed-point arguments. We first construct the linear closed-loop semigroup via a regularization procedure and a Bona-Smith approximation argument, obtaining uniform smoothing estimates of order $5/2$ in the periodic direction. We then prove a unique continuation property, which yields an observability inequality and exponential decay of the linear flow. The nonlinear stabilization result follows from a contraction argument in suitable smoothing spaces, exploiting the bilinear structure of the nonlinearity. At the end, the controllability result is obtained as a perturbation of the linear HUM framework.

Finally, by emphasizing the main features of the present work relative to the existing literature on stabilization and control for dispersive equations.

\begin{itemize}
	\item \textbf{A cylindrical KP setting on \(\mathbb{T}\times\mathbb{R}\).}
	In contrast with the purely periodic framework considered in  \cite{FloreSmith2019} and the full-space settings treated in  \cite{LiXiao2008,LiYanZhang2017}, the present work is carried out on the mixed domain  \(\Omega=\mathbb{T}_x\times\mathbb{R}_y\). This geometry combines discrete frequencies in the longitudinal variable with continuous frequencies in the transverse direction, and naturally leads to the KP-adapted Sobolev scale \(X_s\), tailored to the anisotropic structure of the equation and of the control mechanism.
	\vspace{0.1cm}
	\item \textbf{A unique continuation argument adapted to the KP5 cylinder geometry.}
	A central ingredient of the analysis is the unique continuation property established in Theorem~\ref{Th:UCP-5KP}. For \(\varepsilon>0\), the argument exploits the analyticity generated by the fifth-order dissipative regularization in the periodic variable. For \(\varepsilon=0\), it combines Fourier analysis in the transverse variable with a reduction to a one-dimensional problem on \(\mathbb{T}\), together with a frequency grouping argument and the one-sided Fourier vanishing mechanism used in \cite{LR}.  To the best of our knowledge, this combination is new in the KP control setting.
		\vspace{0.1cm}
	\item \textbf{A fifth-order KP stabilization mechanism.}
	The dispersive operator $\beta\,\partial_x^5+\gamma\,\partial_x^{-1}\partial_y^2$ is coupled with a localized fifth-order feedback \(GD_x^5G\), producing a gain of  \(5/2\) derivatives on the control region and leading to a propagation-of-regularity mechanism in the KP-adapted scale. This extends to the fifth-order KP equation, in a genuinely two-dimensional anisotropic setting, the general strategy developed for one-dimensional dispersive models such as \cite{FloreSmith2019}, while requiring new estimates adapted to the mixed geometry and a careful treatment of operators involving  \(\partial_x^{-1}\).
		\vspace{0.1cm}
	\item \textbf{A direct global small-data theory for the closed-loop problem.}
	For the nonlinear closed-loop equation, we obtain global well-posedness and exponential stability for sufficiently small data in \(X_{s,0}\), \(s>2\). In particular, the global existence statement is formulated explicitly as an independent result within the KP-adapted functional framework, rather than being left only as an implicit consequence of the stabilization argument.
\end{itemize}

\subsection{Outline} The paper is organized as follows. In Section \ref{sec:prelim}, we introduce the functional setting and establish preliminary properties for a regularized version of the system \eqref{aKPe2D}. Section \ref{sec:WLT} is devoted to the local well-posedness of this perturbation and its convergence to the original equation via a Bona-Smith argument. In Section \ref{sec:NWP}, we develop the nonlinear well-posedness framework using KP-adapted slab spaces and bilinear estimates, leading to local and global results. The unique continuation property (Theorem \ref{Th:UCP-5KP}) is proved in Section \ref{sec:UCP}. Section \ref{sec:ucp-obs} establishes the stabilization and controllability results (Theorems \ref{thm:stab5KP} and \ref{thm:control5KP}). Finally, Section \ref{Secfinal} discusses extensions to more general nonlinearities and the challenges of the fully periodic setting $\mathbb{T}^2$.

\section{Preliminaries: spaces, multipliers and commutators}\label{sec:prelim}
In this section, let us introduce the relevant notations and several properties that help us to achieve the main results in this work.

\subsection{KP-adapted Sobolev spaces and the mean-zero constraint in $x$} For $s\in\mathbb{R}$ and $T>0$ we introduce the KP-adapted smoothing space
\begin{equation}\label{eq:ZsT-5KP}
	Z^{s,T}
	:= C([0,T];X_{s,0}(\Omega))
	\cap L^2(0,T;X_{s+5/2,0}(\Omega)),
\end{equation}
endowed with the norm
\begin{equation}\label{eq:Z-norm-5KP}
	\|v\|_{Z^{s,T}}
	:= \|v\|_{L^\infty(0,T;X_{s})}
	+ \|v\|_{L^2(0,T;X_{s+5/2})}.
\end{equation}

When only regularity in $x$ is needed (with $y$ as a parameter), we use the anisotropic space given by
$$
	\mathcal{H}_0^s(\Omega)
	:= \left\{u\in H^s(\mathbb{T}_x;L^2(\mathbb{R}_y)):
	\ \int_{\mathbb{T}}u(x,y)\,dx=0
	\ \text{for a.e.\ }y\right\},
$$
with norm
\[
\|u\|_{\mathcal{H}^s}^2
:= \int_{\mathbb{R}}\sum_{k\in\mathbb{Z}\setminus\{0\}}
(1+k^2)^s\,|\widehat u(k,\eta)|^2\,d\eta.
\]

\subsection{Fourier multipliers in the periodic direction} For $r\in\mathbb{R}$ we define the homogeneous Fourier multiplier $D_x^r$ on nonzero $x$-modes by
$$
	\widehat{(D_x^r u)}(k,\eta)
	= |k|^{r}\,\widehat u(k,\eta),
	\qquad k\in\mathbb{Z}\setminus\{0\},\ \eta\in\mathbb{R}.
$$
On $\mathcal{H}_0^s(\Omega)$ one has the norm equivalence
$$
	\|u\|_{\mathcal{H}^s}
	\cong \|\langle D_x\rangle^s u\|_{L^2(\Omega)}
	\cong \|D_x^{s}u\|_{L^2(\Omega)},
	\qquad u\in\mathcal{H}_0^s(\Omega),
$$
where $\langle D_x\rangle^s$ denotes the inhomogeneous multiplier with symbol $(1+k^2)^{s/2}$ on $k\neq0$. Here $\cong$ denotes equivalence up to constants depending only on $s$.



\subsection{The localization operator \texorpdfstring{$G$}{G}} Let $g\in C^\infty(\mathbb{T})$ be nonnegative, not identically zero, and
normalized by $\int_{\mathbb{T}} g(x)\,dx = 1$. For $f:\mathbb{T}\times\mathbb{R}\to\mathbb{C}$ consider $(Gf)(x,y)$ defined by \eqref{eq:G-def-prelim}. The operator $G$ acts only in the periodic variable $x$ and preserves the zero $x$-mean property. The first lemma gives properties of the localization operator. 

	\begin{lemma}\label{lem:G-properties} Let $s\in\mathbb{R}$. Then:
	\begin{itemize}
		\item[(i)] $G$ maps $X_{s,0}(\Omega)$ into itself and is bounded on
		$X_s(\Omega)$;
		\item[(ii)] $G$ is self-adjoint on $L^2(\Omega)$;
		\item[(iii)] for every $\psi\in C^\infty(\mathbb{T})$ and
		$h\in X_s(\Omega)$,
		\[
		G(\psi h) = \psi\,Gh + \tilde h,
		\]
		with $\|\tilde h\|_{X_s}\le C(s,g,\psi)\|h\|_{X_s}$;
		\item[(iv)] for every $s,r_1,r_2\in\mathbb{R}$, there exists
		$C=C(r_1,r_2,s,g)>0$ such that
		\[
		\|D_x^{r_1}[D_x^s,G]D_x^{r_2}f\|_{L^2(\Omega)}
		\le C\,\|f\|_{X_{r_1+r_2+s-1}(\Omega)}
		\]
		for all $f\in X_{r_1+r_2+s-1,0}(\Omega)$.
	\end{itemize}
\end{lemma}
\begin{proof}

\emph{(i) Invariance and boundedness.}

From the definition
\[
(Gf)(x,y)=g(x)\left(f(x,y)-\int_{\mathbb{T}} g(\tilde x)f(\tilde x,y)\,d\tilde x\right)
\]
and $\dis\int_{\mathbb{T}} g\,dx=1$, we obtain
$\dis\int_{\mathbb{T}} (Gf)(x,y)\,dx=0$,
so $G$ maps $X_{s,0}(\Omega)$ into itself. Since $g$ depends only on $x$,
$G$ acts fiberwise in $y$ as multiplication by $g$ composed with subtraction
of the $g$\nobreakdash-average, which are bounded on $H^s(\mathbb{T})$.
Hence $G$ is bounded on $H^s(\Omega)$:
\begin{equation}\label{eq:G-Hs-bound}
	\|Gf\|_{H^s}\le C(s,g)\,\|f\|_{H^s}.
\end{equation}

To establish boundedness on $X_s(\Omega)$, it remains to control
$\|\partial_x^{-1}(Gu)\|_{H^s}$.  We decompose
\begin{equation}\label{eq:dxinv-G-commutator}
	\partial_x^{-1}(Gu)
	= G(\partial_x^{-1}u) + [\partial_x^{-1},G]\,u.
\end{equation}
The first term is estimated by \eqref{eq:G-Hs-bound}:
\[
\|G(\partial_x^{-1}u)\|_{H^s}
\le C\,\|\partial_x^{-1}u\|_{H^s}.
\]
For the commutator, we write $G=M_g(\mathrm{Id}-\Pi_g)$, where $M_g$
denotes multiplication by $g$ and
$(\Pi_g f)(y):=\dis\int_{\mathbb{T}}g(\tilde x)f(\tilde x,y)\,d\tilde x$.
Then
\[
[\partial_x^{-1},G]
= [\partial_x^{-1},M_g](\mathrm{Id}-\Pi_g)
- M_g\,[\partial_x^{-1},\Pi_g].
\]
By the standard pseudodifferential calculus on \(\mathbb{T}\), the commutator
\([\partial_x^{-1},M_g]\) is an operator of order \(-2\) in the \(x\)-variable.
In particular, for every \(s\in\mathbb{R}\),
$$
	\|[\partial_x^{-1},M_g]f\|_{H^s}
	\le C(s,g)\,\|f\|_{H^{s-2}}
	\le C(s,g)\,\|f\|_{H^{s}}.
$$
For the rank-one part, note that $\Pi_g f$ depends only on $y$, so
$\partial_x^{-1}(\Pi_g f)=0$ (it has no nonzero $x$\nobreakdash-mode).
On the other hand,
$\Pi_g(\partial_x^{-1}f)(y)=\dis\int_{\mathbb{T}}g(\tilde x)\,
\partial_x^{-1}f(\tilde x,y)\,d\tilde x$,
hence
\[
[\partial_x^{-1},\Pi_g]f
= -\Pi_g(\partial_x^{-1}f),
\]
and
$\|M_g\,\Pi_g(\partial_x^{-1}f)\|_{H^s}
\le C(s,g)\,\|\partial_x^{-1}f\|_{L^2}$.
Combining,
\begin{equation}\label{eq:comm-dxinv-G-bound}
	\|[\partial_x^{-1},G]u\|_{H^s}
	\le C(s,g)\left(\|u\|_{H^s}+\|\partial_x^{-1}u\|_{L^2}\right)
	\le C(s,g)\,\|u\|_{X_s}.
\end{equation}
Inserting \eqref{eq:G-Hs-bound} and \eqref{eq:comm-dxinv-G-bound}
into \eqref{eq:dxinv-G-commutator} yields
\[
\|\partial_x^{-1}(Gu)\|_{H^s}
\le C\left(\|\partial_x^{-1}u\|_{H^s}+\|u\|_{X_s}\right)
\le C\,\|u\|_{X_s},
\]
and therefore $\|Gu\|_{X_s}\le C\,\|u\|_{X_s}$.

\medskip
\noindent
\emph{(ii) Self-adjointness.}
For $f,h\in L^2(\Omega)$, a direct computation using Fubini gives
\[
(Gf,h)=\int g(x)(f-\langle f\rangle_g)\overline{h}\,dxdy
      =\int f\,\overline{g(x)(h-\langle h\rangle_g)}\,dxdy
      =(f,Gh),
\]
where $\langle f\rangle_g=\int g(\tilde x)f(\tilde x,y)\,d\tilde x$.

\medskip
\noindent
\emph{(iii) Multipliers.}
For $\psi\in C^\infty(\mathbb{T})$,
\[
G(\psi h)-\psi Gh
= g(x)\int_{\mathbb{T}} g(\tilde x)\left(\psi(x)-\psi(\tilde x)\right)
h(\tilde x,y)\,d\tilde x =: \tilde h\,.
\]
The kernel $K(x,\tilde x)=g(x)(\psi(x)-\psi(\tilde x))g(\tilde x)$
is smooth on $\mathbb{T}^2$, hence the integral operator $T_K$
defined by $\tilde h=T_K h$ is bounded on $H^s(\Omega)$ for every
$s\in\mathbb{R}$:
\[
\|\tilde h\|_{H^s}\le C(s,g,\psi)\,\|h\|_{H^s}.
\]
To obtain the bound in $X_s(\Omega)$, we estimate
$\partial_x^{-1}\tilde h$ by decomposing
\[
\partial_x^{-1}(T_K h)
= T_K(\partial_x^{-1}h) + [\partial_x^{-1},T_K]\,h.
\]
Since $T_K$ is an integral operator in $x$ with smooth kernel,
$\|T_K(\partial_x^{-1}h)\|_{H^s}\le C\,\|\partial_x^{-1}h\|_{H^s}$.
Moreover, $[\partial_x^{-1},T_K]$ is again an integral operator in $x$
whose kernel $K_1(x,\tilde x)$ is obtained from $K$ by commuting with
$\partial_x^{-1}$; because $K$ is smooth and $\partial_x^{-1}$ has
order $-1$, the resulting kernel $K_1$ is smooth, and hence
$\|[\partial_x^{-1},T_K]h\|_{H^s}\le C\,\|h\|_{H^s}$.
Combining,
\[
\|\tilde h\|_{X_s}
\le C(s,g,\psi)\,\|h\|_{X_s}.
\]

\noindent
\emph{(iv) Commutator estimate.}
For each fixed $y$, $G$ is the sum of multiplication by $g$ and a rank–one
operator with smooth kernel in $x$. Standard commutator estimates on
$\mathbb{T}$ imply that $[D_x^s,G]$ has order $s-1$, hence
\[
\|D_x^{r_1}[D_x^s,G]D_x^{r_2}f(\cdot,y)\|_{L_x^2}
\le C\|f(\cdot,y)\|_{H_x^{r_1+r_2+s-1}} .
\]
Integrating in $y$ and arguing as in (i) yields the estimate in
$X_{r_1+r_2+s-1}(\Omega)$.
\end{proof}

\subsection{Regularized linear problem} For $0<\varepsilon<1$ we consider the regularized linear problem
\begin{equation}\label{eq:linear-reg}
	\left\{
	\begin{aligned}
		\partial_t v
		+ \left(\varepsilon D_x^5 - \beta\,\partial_x^5\right)v
		+ \gamma\,\partial_x^{-1}\partial_y^2 v
		+ G D_x^5 G v
		&= F,
		&& (x,y,t)\in\Omega\times(0,T),\\
		v(x,y,0) &= v_0(x,y),
	\end{aligned}
	\right.
\end{equation}
with $v_0\in X_{s,0}(\Omega)$ and $F\in L^2(0,T;X_{s-5/2,0}(\Omega))$. The regularization term $\varepsilon D_x^5 v$ makes the principal part strictly parabolic in the $x$-direction, allowing the use of classical semigroup theory and energy methods. All estimates will be derived with constants independent of $0<\varepsilon<1$, which will later enable us to pass to the limit as $\varepsilon \to 0$.

For $s\in\mathbb{R}$ we write $w=D_x^s v$. Since $D_x^s$ acts only on the periodic variable $x$, it commutes with $\partial_y$ and $\partial_x^{-1}$. Applying $D_x^s$ to \eqref{eq:linear-reg} gives
\begin{equation}\label{eq:w-eq-5KP-prelim}
	\partial_t w
	- \beta\,\partial_x^5 w
	+ \gamma\,\partial_x^{-1}\partial_y^2 w
	+ G D_x^{5} G w
	+ \varepsilon D_x^{5} w
	+ E_s w
	= D_x^s F,
\end{equation}
where the commutator term $E_s$ is defined by
\begin{equation}\label{eq:Es-def}
	E_s
	:= G D_x^{5}\,[D_x^{s},G]\,D_x^{-s}
	+ [D_x^{s},G]\,D_x^{5}\,G\,D_x^{-s}.
\end{equation}

Now, obseve that lemma \ref{lem:G-properties}, item (iv), implies that $E_s$ is one order lower than $D_x^5$ in the $x$–variable, so that, for every $\alpha\in\mathbb{R}$ and $\sigma\in\mathbb{R}$, we get
\begin{equation}\label{eq:Es-order}
	\|D_x^\alpha E_s f\|_{L^2(\Omega)}
	\le C(\alpha,s,g)\,\|f\|_{\mathcal{H}^{\alpha+4}(\Omega)},
\end{equation}
wich yields that $E_s:\mathcal{H}^\sigma(\Omega)\to\mathcal{H}^{\sigma-1}(\Omega)$ continuously. In the smoothing estimate, we will treat $E_s$ as a lower-order perturbation.

The next result is a Kato-type weighted identity for the fifth-order operator, which is a key tool in the propagation-of-regularity argument.

\begin{lemma}\label{lem:weighted-5KP}  Let $v$ be a smooth real-valued solution of \eqref{eq:linear-reg} and set $w:=D_x^s v$. Then for
	every $\psi\in C^\infty(\mathbb{T})$, holds that
	\begin{equation}\label{eq:weighted-5KP}
	\begin{split}
		&\frac12\frac{d}{dt}\int_{\mathbb{R}}\int_{\mathbb{T}} w^2\,\psi\,dx\,dy
		+\frac{5\beta}{2}\int_{\mathbb{R}}\int_{\mathbb{T}} (\partial_x^2 w)^2\,\psi'\,dx\,dy
		+\int_{\mathbb{R}}\int_{\mathbb{T}} D_x^{5/2}(G w)\,D_x^{5/2}G(\psi w)\,dx\,dy\\
		&\quad
		+\varepsilon\left\{
		\int_{\mathbb{R}}\int_{\mathbb{T}} (D_x^{5/2} w)^2\,\psi\,dx\,dy
		+\int_{\mathbb{R}}\int_{\mathbb{T}} D_x^{5/2}w\,[D_x^{5/2},\psi]\,w\,dx\,dy
		\right\}
		+\int_{\mathbb{R}}\int_{\mathbb{T}} w\,(E_s w)\,\psi\,dx\,dy\\
		&=
		\frac{5\beta}{2}\int_{\mathbb{R}}\int_{\mathbb{T}} (\partial_x w)^2\,\psi^{(3)}\,dx\,dy
		-\frac{\beta}{2}\int_{\mathbb{R}}\int_{\mathbb{T}} w^2\,\psi^{(5)}\,dx\,dy
		+\int_{\mathbb{R}}\int_{\mathbb{T}} w\,D_x^sF\,\psi\,dx\,dy\\
		&\quad
		-\frac{\gamma}{2}\int_{\mathbb{R}}\int_{\mathbb{T}} w\,[\partial_x^{-1}\partial_y^2,\psi]\,w\,dx\,dy.
	\end{split}
	\end{equation}
\end{lemma}
\begin{proof}
Multiply \eqref{eq:w-eq-5KP-prelim} by $\psi(x)w$ and integrate over $\Omega=\mathbb{T}\times\mathbb{R}$. The time derivative gives
\[
\int_\Omega (\partial_t w)\psi w
= \frac12\frac{d}{dt}\int_\Omega w^2\psi .
\]
For the dispersive term, integrating by parts in $x$ and using periodicity yields the one-dimensional identity
$$
\int_{\mathbb{T}} (\partial_x^5 w)\psi w\,dx
= -\frac{5}{2}\int_{\mathbb{T}} (\partial_x^2 w)^2\psi'\,dx
+ \frac{5}{2}\int_{\mathbb{T}} (\partial_x w)^2\psi^{(3)}\,dx
- \frac{1}{2}\int_{\mathbb{T}} w^2\psi^{(5)}\,dx .
$$
Integrating in $y$ gives the corresponding contribution to \eqref{eq:weighted-5KP}. Now, for the feedback term, using that $G$ is self-adjoint in $L^2(\Omega)$ and
$D_x^5=(D_x^{5/2})^\ast D_x^{5/2}$ on the mean-zero subspace,
\[
\int_\Omega (G D_x^5 G w)\psi w
= \int_\Omega (D_x^{5/2}G w)\,D_x^{5/2}G(\psi w).
\]
Similarly, by self-adjointness of $D_x^{5/2}$,
\[
\varepsilon\int_\Omega (D_x^5 w)\psi w
= \varepsilon\int_\Omega (D_x^{5/2}w)
\left(\psi D_x^{5/2}w+[D_x^{5/2},\psi]w\right).
\]
The term involving $E_s$ remains as $\int_\Omega w(E_s w)\psi$. For the transverse operator $L:=\partial_x^{-1}\partial_y^2$, note that
$\partial_x^{-1}$ is skew-adjoint and $\partial_y^2$ is self-adjoint on the $x$–mean-zero subspace, hence $L^\ast=-L$. For real-valued $w$,
\[
\int_\Omega (Lw)\psi w
= -\frac12\int_\Omega w[L,\psi]w .
\]
Finally, the forcing term contributes $\int_\Omega D_x^sF\,\psi w$. Collecting all terms yields \eqref{eq:weighted-5KP}.
\end{proof}

\begin{remark}[Transverse commutator]\label{rem:transverse-comm}
	Since $\psi=\psi(x)$ is independent of $y$,
	\[
	[\partial_x^{-1}\partial_y^2,\psi]
	= [\partial_x^{-1},\psi]\,\partial_y^2.
	\]
	The operator $[\partial_x^{-1},\psi]$ is a pseudodifferential operator of order $-2$ in $x$ with
	smooth coefficients, so $[\partial_x^{-1}\partial_y^2,\psi]$ has overall order $0$ in $(x,y)$.
	Consequently, for smooth $w$,
	\begin{equation}\label{eq:transverse-comm-bound}
		\left|\int_{\mathbb{R}}\int_{\mathbb{T}} w\,[\partial_x^{-1}\partial_y^2,\psi]\,w\,dx\,dy\right|
		\le C(\psi)\,\|w\|_{X_0}^2,
	\end{equation}
	where $C(\psi)$ depends on finitely many derivatives of $\psi$. In particular, the sign in \eqref{eq:weighted-5KP} does not affect the order-of-magnitude bound
	\eqref{eq:transverse-comm-bound}, and this term will be treated as lower order in the
	propagation argument.
\end{remark}


The weighted identity allows us to propagate the regularity generated in the control region $\omega$ by the feedback $G D_x^5 G$. The argument proceeds in two stages: first in the anisotropic $x$–Sobolev scale $\mathcal{H}_0^s$, then in the KP–adapted scale $X_{s,0}$.

\begin{proposition}[Propagation of regularity]\label{prop:propagation}
	Let $s\in\mathbb{R}$, $T>0$, $v_0\in X_{s,0}(\Omega)$, and
	\[
	F\in L^2\left(0,T;X_{s-5/2,0}(\Omega)\right).
	\]
	For $0<\varepsilon<1$, let $v$ be the solution of \eqref{eq:linear-reg} on $[0,T]$ with data
	$(v_0,F)$. Then
	\begin{equation}\label{eq:propagation-est-again}
		\|v\|_{L^2(0,T;X_{s+5/2})}
		\le C(s,T)\left(\|v_0\|_{X_s}
		+ \|F\|_{L^2(0,T;X_{s-5/2})}\right),
	\end{equation}
	for some constant $C(s,T)>0$ nondecreasing in $T$ and independent of $\varepsilon\in(0,1)$.
	Equivalently,
	\[
	\|v\|_{Z^{s,T}}
	\le C(s,T)\left(\|v_0\|_{X_s}
	+ \|F\|_{L^2(0,T;X_{s-5/2})}\right).
	\]
\end{proposition}
\begin{proof}
Throughout the proof, $C(\cdot)$ denotes constants that may change from line to line but are independent of $0<\varepsilon<1$. Let us give the proof step by step. 

\medskip
\noindent\emph{\textbf{Step 1.} Reduction to smooth data.}  

\medskip

Assume first that $v_0$ and $F$ are smooth and compactly supported in time.
Then the solution $v$ of \eqref{eq:linear-reg} is smooth on $[0,T]\times\Omega$. The general case follows by density in $X_{s,0}(\Omega)$ and $L^2(0,T;X_{s-5/2,0}(\Omega))$.

\medskip
\noindent\emph{\textbf{Step 2.} Equation for $w=D_x^s v$.}  
\medskip

Setting $w:=D_x^s v$ and applying $D_x^s$ to \eqref{eq:linear-reg}, the
commutator decomposition \eqref{eq:w-eq-5KP-prelim}–\eqref{eq:Es-def} yields
\begin{equation}\label{eq:w-eq-prop-KP-final-short}
\partial_t w
- \beta\,\partial_x^5 w
+ \gamma\,\partial_x^{-1}\partial_y^2 w
+ G D_x^{5} G w
+ \varepsilon D_x^{5} w
+ E_s w
= D_x^s F,
\end{equation}
where $E_s$ is an order-$4$ operator satisfying \eqref{eq:Es-order}.

\medskip
\noindent\emph{\textbf{Step 3.} Weighted energy identity.} 
\medskip

Fix $\psi\in C^\infty(\mathbb{T})$. Applying Lemma \ref{lem:weighted-5KP} to
$w$, integrating in time on $(0,\tau)$, and moving the non-negative terms to the left-hand side gives
\begin{equation}\label{eq:prop-intermediate-short}
\begin{aligned}
&\frac12\int_\Omega w(\tau)^2\psi
+ \frac{5\beta}{2}\int_0^\tau\!\!\int_\Omega (\partial_x^2 w)^2\psi'
+ \int_0^\tau\!\!\int_\Omega D_x^{5/2}(Gw)D_x^{5/2}G(\psi w)
+ \varepsilon\int_0^\tau\!\!\int_\Omega (D_x^{5/2}w)^2\psi
\\
&\le
\frac12\int_\Omega w(0)^2\psi
+ C(\psi)\!\int_0^\tau\!\!\int_\Omega (|w|^2+|\partial_x w|^2)
+ \int_0^\tau\!\!\int_\Omega |w||D_x^sF|
\\
&\quad
+ \int_0^\tau\!\!\int_\Omega |w||E_sw|
+ \varepsilon\!\int_0^\tau\!\!\int_\Omega |D_x^{5/2}w[D_x^{5/2},\psi]w|
+ \left|\int_0^\tau\!\!\int_\Omega w[\partial_x^{-1}\partial_y^2,\psi]w\right|.
\end{aligned}
\end{equation}
The terms involving $\psi^{(3)}$ and $\psi^{(5)}$ have been absorbed in the
lower-order contribution.

\medskip
\noindent\emph{\textbf{Step 4.} Lower-order terms.} 
\medskip

Each term on the right-hand side is controlled by $\|w\|_{L^2(0,T;\mathcal H^{5/2})}$ and $\|w\|_{L^2(0,T;\mathcal H^0)}$. Indeed,
\[
\int |w||D_x^sF|
\le C(\psi)\|w\|_{L^2(0,T;\mathcal H^{5/2})}
\|F\|_{L^2(0,T;\mathcal H^{s-5/2})},
\]
by $\mathcal H^{5/2}$–$\mathcal H^{-5/2}$ duality. Using \eqref{eq:Es-order} and $\|w\|_{\mathcal H^{3/2}}\le \|w\|_{\mathcal H^{5/2}}$,
\[
\int |w||E_sw|
\le C(\psi)\int_0^T\|w\|_{\mathcal H^{5/2}}^2.
\]
Since $[D_x^{5/2},\psi]$ has order $3/2$,
\[
\varepsilon\int |D_x^{5/2}w[D_x^{5/2},\psi]w|
\le C(\psi)\int_0^T\|w\|_{\mathcal H^{5/2}}^2,
\]
uniformly in $\varepsilon$. Finally, by Remark \ref{rem:transverse-comm},
$[\partial_x^{-1}\partial_y^2,\psi]$ has order $0$, hence
\[
\left|\int w[\partial_x^{-1}\partial_y^2,\psi]w\right|
\le C(\psi)\int_0^T\|w\|_{\mathcal H^0}^2 .
\]

\noindent\emph{\textbf{Step 5.} $\mathcal H^s$-estimates}
\medskip

Let $s=0$. Taking $\psi\equiv1$ in Lemma \ref{lem:weighted-5KP} eliminates all commutators and derivative terms. Using Cauchy-Schwarz, Young's inequality, and the coercivity of the feedback operator,
\[
\|D_x^{5/2}(Gw)\|_{L^2}^2
\ge c\|w\|_{\mathcal H^{5/2}}^2 - C\|w\|_{\mathcal H^0}^2,
\]
one obtains, after Grönwall inequality,
\begin{equation}\label{eq:L2-energy-short}
\|w\|_{L^\infty(0,T;\mathcal H^0)}^2
+ \varepsilon\|w\|_{L^2(0,T;\mathcal H^{5/2})}^2
\le
C(s,T)\left(
\|w_0\|_{\mathcal H^0}^2
+ \|F\|_{L^2(0,T;\mathcal H^{s-5/2})}^2
\right).
\end{equation}

For $s>0$, combining the bounds above with \eqref{eq:prop-intermediate-short}, applying Young's inequality, and using \eqref{eq:L2-energy-short} to control the $\mathcal H^0$ terms yields
\[
\|w\|_{L^\infty(0,T;\mathcal H^0)}
+
\|w\|_{L^2(0,T;\mathcal H^{5/2})}
\le
C(s,T)\left(
\|w_0\|_{\mathcal H^0}
+
\|F\|_{L^2(0,T;\mathcal H^{s-5/2})}
\right).
\]
Since $w=D_x^s v$, this is equivalent to
$$
\|v\|_{L^\infty(0,T;\mathcal H^s)}
+
\|v\|_{L^2(0,T;\mathcal H^{s+5/2})}
\le
C(s,T)\left(
\|v_0\|_{\mathcal H^s}
+
\|F\|_{L^2(0,T;\mathcal H^{s-5/2})}
\right).
$$

Now, let $z:=\partial_x^{-1}v$. Applying $\partial_x^{-1}$ to \eqref{eq:linear-reg} yields an equation for $z$ analogous to \eqref{eq:w-eq-prop-KP-final-short}, up to a commutator $R=[\partial_x^{-1},GD_x^5G]$, which is of order $4$ in $z$ and can be treated as a lower-order perturbation exactly as for $E_s$. Repeating the preceding argument gives
\[
\|z\|_{L^\infty(0,T;\mathcal H^s)}
+
\|z\|_{L^2(0,T;\mathcal H^{s+5/2})}
\le
C(s,T)\left(
\|\partial_x^{-1}v_0\|_{\mathcal H^s}
+
\|\partial_x^{-1}F\|_{L^2(0,T;\mathcal H^{s-5/2})}
\right).
\]
Using the definition of $X_s$ and summing the bounds for $v$ and $z$ yields
\[
\|v\|_{Z^{s,T}}
\le
C(s,T)\left(
\|v_0\|_{X_s}
+
\|F\|_{L^2(0,T;X_{s-5/2})}
\right),
\]
with $C(s,T)$ independent of $0<\varepsilon<1$. This completes the proof.
\end{proof}

The next proposition ensures a $X_s$-energy estimate for the regularized flow. 

\begin{proposition}\label{prop:Hs-estimate-prelim}
	Let $s\in\mathbb{R}$, $T>0$, $0<\varepsilon<1$, $v_0\in X_{s,0}(\Omega)$, and $F\in L^2\left(0,T;X_{s-5/2,0}(\Omega)\right).$ Consider $v$ be a smooth solution of the regularized linear problem \eqref{eq:linear-reg} on $[0,T]$. Then
	\begin{equation}\label{eq:Hs-energy-5KP}
	\begin{split}
		\|v(t)\|_{X_s}^2
		+ \int_0^t\Bigl(&\|D_x^{5/2}(G D_x^s v(\tau))\|_{L^2}^2
		+ \|D_x^{5/2}(G D_x^s z(\tau))\|_{L^2}^2  \\
		&+ \varepsilon\|D_x^{5/2} v(\tau)\|_{X_s}^2\Bigr)\,d\tau
		\le C(s,T)\Bigl(\|v_0\|_{X_s}^2
		+ \|F\|_{L^2(0,T;X_{s-5/2})}^2\Bigr).
	\end{split}
	\end{equation}
	for all $t\in[0,T]$, where $z:=\partial_x^{-1}v$, and $C(s,T)>0$ is nondecreasing in $T$
	and independent of $\varepsilon$.
\end{proposition}

\begin{proof}
Throughout the proof, the constants $C(\cdot)$ may change from line to line but remain independent of $0<\varepsilon<1$. As in Proposition \ref{prop:propagation}, we first work in $\mathcal{H}_0^s(\Omega)$ and then lift the estimate to $X_{s,0}(\Omega)$ using $z=\partial_x^{-1}v$. First, taking the $L^2(\Omega)$ inner product of \eqref{eq:linear-reg} with $v$ yields
\[
\frac12\frac{d}{dt}\|v(t)\|_{L^2}^2
+ \varepsilon(D_x^5 v,v)_{L^2}
+ (G D_x^5 G v,v)_{L^2}
= (F,v)_{L^2},
\]
since $\partial_x^5$ and $\partial_x^{-1}\partial_y^2$ are skew-adjoint on the
$x$–mean-zero subspace. Using the factorization
$D_x^5=(D_x^{5/2})^\ast D_x^{5/2}$ and the self-adjointness of $G$,
\[
(D_x^5 v,v)_{L^2}=\|D_x^{5/2}v\|_{L^2}^2,
\qquad
(GD_x^5Gv,v)_{L^2}=\|D_x^{5/2}(Gv)\|_{L^2}^2,
\]
so that
\[
\frac12\frac{d}{dt}\|v(t)\|_{L^2}^2
+ \varepsilon\|D_x^{5/2}v(t)\|_{L^2}^2
+ \|D_x^{5/2}(Gv(t))\|_{L^2}^2
= (F(t),v(t))_{L^2}.
\]
By Cauchy–Schwarz and Young's inequality, and Grönwall's inequality gives
$$
\|v(t)\|_{L^2}^2
+ \int_0^t\!\bigl(\|D_x^{5/2}(Gv(\tau))\|_{L^2}^2
+ \varepsilon\|D_x^{5/2}v(\tau)\|_{L^2}^2\bigr)\,d\tau
\le C(T)\Bigl(\|v_0\|_{L^2}^2+\|F\|_{L^2(0,T;L^2)}^2\Bigr).
$$

Now, let $w:=D_x^s v$. Applying $D_x^s$ to \eqref{eq:linear-reg} yields
$$
\partial_t w
- \beta\partial_x^5 w
+ \gamma\partial_x^{-1}\partial_y^2 w
+ G D_x^5 G w
+ \varepsilon D_x^5 w
+ E_s w
= D_x^sF,
$$
where $E_s$ is a commutator of order $4$ in $x$. Taking the $L^2$ inner product with $w$ and using the skew-adjointness of $\partial_x^5$ and $\partial_x^{-1}\partial_y^2$ gives
\[
\frac12\frac{d}{dt}\|w(t)\|_{L^2}^2
+ \varepsilon\|D_x^{5/2}w(t)\|_{L^2}^2
+ \|D_x^{5/2}(G w(t))\|_{L^2}^2
= (D_x^sF,w)_{L^2}-(E_s w,w)_{L^2}.
\]
Using $\mathcal{H}^{5/2}$–$\mathcal{H}^{-5/2}$ duality,
\[
|(D_x^sF,w)|\le
\frac{\delta}{2}\|w\|_{\mathcal{H}^{5/2}}^2
+\frac{1}{2\delta}\|F\|_{\mathcal{H}^{s-5/2}}^2,
\]
and
\[
|(E_sw,w)|\le C(s,g)\|w\|_{\mathcal{H}^{5/2}}^2.
\]
Choosing $\delta>0$ small and using the feedback coercivity
$$
\|D_x^{5/2}(Gw)\|_{L^2}^2
\ge c\|w\|_{\mathcal{H}^{5/2}}^2-C\|w\|_{L^2}^2,
$$
we obtain
\[
\frac{d}{dt}\|w(t)\|_{L^2}^2
+ \|D_x^{5/2}(G w(t))\|_{L^2}^2
+ \varepsilon\|D_x^{5/2}w(t)\|_{L^2}^2
\le C(s)\bigl(\|w(t)\|_{L^2}^2+\|F(t)\|_{\mathcal{H}^{s-5/2}}^2\bigr).
\]
Integration and Grönwall yield
\begin{equation}\label{eq:Hs-energy-w-final}
\begin{split}
\|w(t)\|_{L^2}^2
&+ \int_0^t\!\bigl(\|D_x^{5/2}(G w(\tau))\|_{L^2}^2
+ \varepsilon\|D_x^{5/2}w(\tau)\|_{L^2}^2\bigr)\,d\tau\\
&\le C(s,T)\Bigl(\|w_0\|_{L^2}^2
+ \|F\|_{L^2(0,T;\mathcal{H}^{s-5/2})}^2\Bigr).
\end{split}
\end{equation}

Finally, consider $z:=\partial_x^{-1}v$. Applying $\partial_x^{-1}$ to \eqref{eq:linear-reg} shows that $z$ satisfies
$$
\partial_t z
- \beta\partial_x^5 z
+ \gamma\partial_x^{-1}\partial_y^2 z
+ G D_x^5 G z
+ \varepsilon D_x^5 z
+ Rv
= \partial_x^{-1}F,
$$
where $R=[\partial_x^{-1},GD_x^5G]$ is a lower-order operator. Repeating the previous energy argument for $D_x^s z$ yields
\begin{equation}\label{eq:Hs-energy-z-final}
\begin{split}
\|D_x^s z(t)\|_{L^2}^2+& \int_0^t\!\bigl(\|D_x^{5/2}(G D_x^s z)\|_{L^2}^2
+ \varepsilon\|D_x^{5/2}D_x^s z\|_{L^2}^2\bigr)\,d\tau\\
&\le C(s,T)\Bigl(\|\partial_x^{-1}v_0\|_{\mathcal{H}^s}^2
+ \|\partial_x^{-1}F\|_{L^2(0,T;\mathcal{H}^{s-5/2})}^2\Bigr).
\end{split}
\end{equation}
Adding \eqref{eq:Hs-energy-w-final} and \eqref{eq:Hs-energy-z-final} and using the definition of the $X_s$ norm yields
\begin{equation*}
\begin{split}
\|v(t)\|_{X_s}^2
+ \int_0^t\Bigl(&\|D_x^{5/2}(G D_x^s v)\|_{L^2}^2
+ \|D_x^{5/2}(G D_x^s z)\|_{L^2}^2\\
& +\varepsilon\|D_x^{5/2}v\|_{X_s}^2\Bigr)d\tau
\le C(s,T)\Bigl(\|v_0\|_{X_s}^2
+ \|F\|_{L^2(0,T;X_{s-5/2})}^2\Bigr),
\end{split}
\end{equation*}
for all $t\in[0,T]$, with $C(s,T)$ independent of $\varepsilon$. This proves
\eqref{eq:Hs-energy-5KP}.
\end{proof}

\section{Local theory: Well-posedness}\label{sec:WLT} 

This section develops the linear framework for the KP-type model, establishing well-posedness for the associated linear problem and deriving key smoothing and propagation properties. We characterize the generated semigroup, obtain uniform energy estimates in the KP-adapted setting, and use them to control solutions in $Z^{s,T}$.

\subsection{Semigroup properties}

Recall the regularized linear problem \eqref{eq:linear-reg} and the KP-adapted space $Z^{s,T}$ defined in \eqref{eq:ZsT-5KP}-\eqref{eq:Z-norm-5KP}. Combining the basic $X_s$-energy estimate for the regularized flow (Proposition \ref{prop:Hs-estimate-prelim}) with the propagation estimate \eqref{eq:propagation-est-again} from Proposition \ref{prop:propagation}, we obtain uniform control of $v$ in $Z^{s,T}$.

\begin{proposition}\label{prop:uniform-Zs} 	Let $s\in\mathbb{R}$, $T>0$, and $0<\varepsilon<1$. Consider $v$ be a smooth solution of \eqref{eq:linear-reg} on $[0,T]$ with $v_0\in X_{s,0}(\Omega)$ and $F\in L^2\bigl(0,T;X_{s-5/2,0}(\Omega)\bigr).$ Then $v\in Z^{s,T}$ and
	\begin{equation}\label{eq:Zs-estimate}
		\|v\|_{Z^{s,T}}
		\le C(s,T)\bigl(\|v_0\|_{X_s} + \|F\|_{L^2(0,T;X_{s-5/2})}\bigr),
	\end{equation}
	for some constant $C(s,T)>0$ nondecreasing in $T$ and independent of $\varepsilon$.
\end{proposition}

\begin{proof} Note that Proposition \ref{prop:Hs-estimate-prelim} yields the uniform $X_s$–energy bound
\begin{align}\label{eq:Xs-ctrl-uniform}
\|v\|_{L^\infty(0,T;X_s)}^2
+ \int_0^T\Bigl(&\|D_x^{5/2}(GD_x^s v(t))\|_{L^2}^2
+ \|D_x^{5/2}(GD_x^s z(t))\|_{L^2}^2  \notag\\
&+ \varepsilon\|D_x^{5/2}v(t)\|_{X_s}^2\Bigr)\,dt
\le C(s,T)\Bigl(\|v_0\|_{X_s}^2
+ \|F\|_{L^2(0,T;X_{s-5/2})}^2\Bigr),
\end{align}
where $z=\partial_x^{-1}v$ and $C(s,T)$ is independent of $0<\varepsilon<1$.
Moreover, Proposition \ref{prop:propagation} gives
\[
\|v\|_{L^2(0,T;X_{s+5/2})}
\le C(s,T)\Bigl(\|v_0\|_{X_s} + \|F\|_{L^2(0,T;X_{s-5/2})}\Bigr),
\]
with the same independence in $\varepsilon$. Taking square roots in \eqref{eq:Xs-ctrl-uniform} and combining the two estimates, \eqref{eq:Zs-estimate} holds.
\end{proof}

We now turn to the semigroup generated by the regularized closed-loop operator
$$
	L_\varepsilon
	:= -\beta\,\partial_x^5 + \gamma\,\partial_x^{-1}\partial_y^2 + G D_x^5 G + \varepsilon D_x^5,
	\qquad
	\mathcal{D}(L_\varepsilon)=X_{5,0}(\Omega),
$$
acting on the KP–adapted energy space $X_{0,0}(\Omega)$, so that \eqref{eq:linear-reg} reads $\partial_t v + L_\varepsilon v = F$.

\begin{proposition}\label{prop:semigroup-5KP} Let $\varepsilon>0$. Then $-L_\varepsilon$ generates an analytic semigroup  $\{\mathcal{S}_\varepsilon(t)\}_{t\ge0}$ on $X_{0,0}(\Omega)$.  Moreover, for every $s\ge0$ the semigroup restricts to $X_{s,0}(\Omega)$ and
	\[
	\mathcal{S}_\varepsilon(t)X_{s,0}(\Omega)\subset X_{s,0}(\Omega),\qquad \forall\,t\ge0.
	\]
\end{proposition}

\begin{proof} Write
$$
L_\varepsilon = A_\varepsilon + B,
\qquad
A_\varepsilon := \varepsilon D_x^5,
\qquad\text{and}\quad
B := -\beta\,\partial_x^5 + \gamma\,\partial_x^{-1}\partial_y^2 + G D_x^5 G,
$$
with common domain $\mathcal{D}(A_\varepsilon)=\mathcal{D}(B)=X_{5,0}(\Omega)$. The proof proceeds in three steps: sectoriality of $A_\varepsilon$, sectoriality of $L_\varepsilon$, and invariance of $X_{s,0}$.

\medskip
\noindent\emph{\textbf{Step 1.} Sectoriality of $A_\varepsilon$ on $X_{0,0}$.} 
\medskip

On the $x$–mean-zero subspace, $D_x^5$ has symbol $|k|^5$, hence $A_\varepsilon$ is positive self-adjoint on $L^2(\Omega)$ and
$$
(A_\varepsilon f,f)_{L^2}
= \varepsilon\|D_x^{5/2}f\|_{L^2}^2 \ge 0,
\qquad \forall f\in X_{5,0}(\Omega).
$$
Thus $-A_\varepsilon$ generates a contraction semigroup on $L^2(\Omega)$. Since $D_x^5$ and $\partial_x^{-1}$ are Fourier multipliers on nonzero $x$–modes, they commute, hence $A_\varepsilon$ preserves $X_{0,0}(\Omega)$ and its resolvent $(\lambda+A_\varepsilon)^{-1}$ preserves the $x$–mean-zero constraint. Consequently $-A_\varepsilon$ generates an analytic contraction semigroup on $X_{0,0}(\Omega)$. For $\mathrm{Re}\,\lambda>0$, we have
\begin{equation}\label{eq:resolvent-Aeps-X0}
\|(\lambda + A_\varepsilon)^{-1}\|_{\mathcal{L}(X_0)}
\le \frac{1}{\mathrm{Re}\,\lambda},
\qquad\text{and}\qquad
\|A_\varepsilon(\lambda + A_\varepsilon)^{-1}\|_{\mathcal{L}(X_0)}
\le 1,
\end{equation}
since the symbol of $A_\varepsilon(\lambda+A_\varepsilon)^{-1}$ is $\varepsilon|k|^5/(\lambda+\varepsilon|k|^5)$, which is bounded by $1$. In particular
$$\|D_x^5(\lambda+A_\varepsilon)^{-1}\|_{\mathcal{L}(X_0)}\le C\varepsilon^{-1}.$$

\medskip
\noindent\emph{\textbf{Step 2.} Sectoriality of $L_\varepsilon = A_\varepsilon + B$.} 
\medskip

Decompose $B=B_{\mathrm{s}}+B_+$ with
\[
B_{\mathrm{s}}:=-\beta\,\partial_x^5+\gamma\,\partial_x^{-1}\partial_y^2,
\qquad\text{and}\qquad
B_+:=G D_x^5 G.
\]
Here, $B_{\mathrm{s}}$ is skew-adjoint on $L^2(\Omega)$ and $B_+$ is positive self-adjoint. Hence,
\begin{equation}\label{eq:Re-Leps}
\mathrm{Re}(L_\varepsilon f,f)_{L^2}
=\varepsilon\|D_x^{5/2}f\|_{L^2}^2
+\|D_x^{5/2}(Gf)\|_{L^2}^2
\ge \varepsilon\|D_x^{5/2}f\|_{L^2}^2,
\end{equation}
 for $f\in X_{5,0}(\Omega)$. Thus $L_\varepsilon$ is accretive.

For the imaginary part, note that
\[
\mathrm{Im}(L_\varepsilon f,f)_{L^2}=\mathrm{Im}(B_{\mathrm{s}}f,f)_{L^2}.
\]
Since $\partial_x^5$ has symbol $ik^5$ on the mean-zero subspace, we get
\[
|(\partial_x^5 f,f)_{L^2}|
\le \|D_x^{5/2}f\|_{L^2}^2 .
\]
Moreover $\partial_x^{-1}\partial_y^2$ is of order $2$, so by interpolation, for any $\delta>0$, yields that
\[
|(\partial_x^{-1}\partial_y^2 f,f)_{L^2}|
\le \delta\|D_x^{5/2}f\|_{L^2}^2 + C_\delta\|f\|_{X_0}^2 .
\]
Combining the previous inequality with \eqref{eq:Re-Leps} gives
\begin{equation*}
\begin{split}
|\mathrm{Im}(L_\varepsilon f,f)_{L^2}|
&\le (|\beta|+C\delta)\|D_x^{5/2}f\|_{L^2}^2
+ C_\delta|\gamma|\|f\|_{X_0}^2 \nonumber\\
&\le \frac{|\beta|+C\delta}{\varepsilon}\mathrm{Re}(L_\varepsilon f,f)_{L^2}
+ C_\delta\|f\|_{X_0}^2 .
\end{split}
\end{equation*}
Hence, the numerical range of $L_\varepsilon\in\Sigma_{\theta_\varepsilon}$ (up to a bounded shift), with $\theta_\varepsilon=\arctan((|\beta|+1)/\varepsilon)$.

To obtain the resolvent bound, we estimate $B$ relative to $A_\varepsilon$. On mean-zero functions, we have that
\[
\|\beta\,\partial_x^5 f\|_{X_0}
\le \frac{|\beta|}{\varepsilon}\|A_\varepsilon f\|_{X_0}.
\]
Since $\partial_x^{-1}\partial_y^2$ is of order $2$, 
\[
\|\gamma\,\partial_x^{-1}\partial_y^2 f\|_{X_0}
\le \delta\|A_\varepsilon f\|_{X_0} + C_{\delta,\gamma,\varepsilon}\|f\|_{X_0}.
\]
For the feedback term, writing $D_x^5(Gf)=G(D_x^5 f)+[D_x^5,G]f$ and noting that $[D_x^5,G]$ has order $4$, holds that
\[
\|G D_x^5 G f\|_{X_0}
\le \frac{C(g)}{\varepsilon}\|A_\varepsilon f\|_{X_0}
+ C_{\delta,g,\varepsilon}\|f\|_{X_0}.
\]
Thus, $B$ is $A_\varepsilon$-bounded with finite relative bound given by
$$
\|Bf\|_{X_0}\le \mu_\varepsilon\|A_\varepsilon f\|_{X_0} + C_\varepsilon\|f\|_{X_0},\qquad \forall f\in X_{5,0}(\Omega),
$$
where $\mu_\varepsilon=C(|\beta|,g)/\varepsilon$. Thanks to \eqref{eq:resolvent-Aeps-X0} follows that
\[
\|B(\lambda+A_\varepsilon)^{-1}\|
\le \mu_\varepsilon + \frac{C_\varepsilon}{|\lambda|}
\to \mu_\varepsilon, \quad \text{as } |\lambda|\to\infty.
\]
For $|\lambda|$ large, the Neumann series yields
\[
(\lambda+L_\varepsilon)^{-1}
=(\lambda+A_\varepsilon)^{-1}
(\mathrm{Id}+B(\lambda+A_\varepsilon)^{-1})^{-1}.
\]
Combining this with the sector estimate and the standard sectoriality criterion \cite[Ch. II, \S4]{EngelNagel2000}, $L_\varepsilon$ is sectorial on $X_{0,0}(\Omega)$ and $-L_\varepsilon$ generates an analytic semigroup $\{\mathcal{S}_\varepsilon(t)\}_{t\ge0}$.

\medskip
\noindent\emph{\textbf{Step 3.} Invariance of $X_{s,0}$.}
\medskip

Let $\Lambda^s=\langle D_x\rangle^s$. The operators $\partial_x^5$, $D_x^5$, $\partial_x^{-1}\partial_y^2$ and $\partial_x^{-1}$
commute with $\Lambda^s$, so the only commutator comes from the feedback term:
\[
[\Lambda^s,L_\varepsilon]
=[\Lambda^s,G D_x^5 G]
=[\Lambda^s,G]D_x^5 G + G D_x^5[\Lambda^s,G].
\]
Since $[\Lambda^s,G]$ has order $s-1$ in $x$, $[\Lambda^s,L_\varepsilon]$ is of lower order and maps $X_{s,0}$ into $X_{s-1,0}$. Standard regularity results for analytic semigroups on sectorial operators, see for instance  \cite[Ch. V]{AmannLinear1995}, implies that $X_{s,0}$ is invariant under $\mathcal{S}_\varepsilon(t)$ for every $s\ge0$.
\end{proof}

\subsection{Passage to the limit \texorpdfstring{$\varepsilon\to0$}{epsilon to 0}}  We will establish the well-posedness and exponential bounds for the semigroup associated with the non-regularized closed-loop linear fifth-order KP system in the KP–adapted Sobolev scale $X_s$. The argument follows the Bona–Smith approximation method \cite{BoSm1975}, with all estimates carried out simultaneously for $v$ and $\partial_x^{-1}v$.

\begin{proposition}\label{prop:BS-limit-5KP-X}  Fix $s\in\mathbb{R}$ and $T>0$. Let $v_0\in X_{s,0}(\Omega)$ and $F\in L^2(0,T;X_{s-5/2}(\Omega))$. Then there exists a unique solution $v\in Z^{s,T}$ to the unregularized closed-loop problem
	\begin{equation}\label{eq:limit-closed-loop-5KP-X}
		\left\{
		\begin{aligned}
			\partial_t v
			- \beta\,\partial_x^5 v
			+ \gamma\,\partial_x^{-1}\partial_y^2 v
			+ G D_x^5 G v &= F,
			&& (x,y,t)\in\Omega\times(0,T),\\
			v(x,y,0) &= v_0(x,y),
		\end{aligned}
		\right.
	\end{equation}
	and
	\begin{equation}\label{eq:apriori-limit-5KP-X}
		\|v\|_{Z^{s,T}}
		\le C(s,T)\Bigl(\|v_0\|_{X_s}
		+ \|F\|_{L^2(0,T;X_{s-5/2})}\Bigr),
	\end{equation}
	with $C(s,T)$ nondecreasing in $T$. Moreover, if $s\ge0$ and $F\equiv0$, then \eqref{eq:limit-closed-loop-5KP-X} generates a  $C_0$–semigroup $\mathcal{S}(t)$ on $X_{s,0}(\Omega)$ and there exist constants $C=C(s)>0$ and $\lambda>0$ such that
	\begin{equation}\label{eq:decay-limit-5KP-X}
		\|\mathcal{S}(t)v_0\|_{X_s}
		\le C e^{-\lambda t}\|v_0\|_{X_s},
		\qquad \forall t\ge0,\quad \text{and}\quad v_0\in X_{s,0}(\Omega).
	\end{equation}
\end{proposition}

\begin{proof}  The proof is divided into three main parts: smoothing of the data and forcing, a Bona–Smith argument in $Z^{s,T}$ for the regularized flows, and the passage to the limit $\varepsilon\to0$ in both the well-posedness and the exponential bounds.
	
\medskip\noindent
	\textbf{i. Smoothing of the data and uniform bounds.}  
	\medskip
	
	Let $v_0\in X_{s,0}(\Omega)$ and $F\in L^2(0,T;X_{s-5/2}(\Omega))$.  
For $\varepsilon\in(0,1)$, define the smoothed initial data by the Gaussian cutoff in the periodic frequency $k$,
$$
\widehat{v_0^\varepsilon}(k,\eta) = e^{-\varepsilon^{1/10}k^2}\widehat{v_0}(k,\eta), \qquad (k,\eta)\in\mathbb{Z}\times\mathbb{R}.
$$
Since the multiplier acts only in $x$ and commutes with $\partial_x^{-1}$, the usual $H^s$ estimate applies to both $v_0^\varepsilon$ and $\partial_x^{-1}v_0^\varepsilon$, so that $v_0^\varepsilon\in X_{\sigma,0}(\Omega)$ for every $\sigma\in\mathbb{R}$. In particular, for $\sigma\ge0$, 
\begin{equation}\label{eq:BS-smoothing-est-5KP-X}
\varepsilon^{\sigma/10}\|v_0^\varepsilon\|_{X_{s+\sigma}}
\le C(\sigma)\|v_0\|_{X_s},
\qquad
\|v_0^\varepsilon-v_0\|_{X_s}\to0, \quad \text{as }\varepsilon\to0.
\end{equation}

Fix a decreasing sequence $\{\varepsilon_n\}\subset(0,1)$ with $\varepsilon_n\to0$ and set $v_0^n:=v_0^{\varepsilon_n}$. Then $v_0^n\to v_0$ in $X_s$. Choose $F^n\in C([0,T];X_{\infty,0}(\Omega))$ such that
\begin{equation}\label{eq:Fn-approx-X}
F^n\to F
\quad \text{in } L^2(0,T;X_{s-5/2}(\Omega)).
\end{equation}
For each $n$, let $v^n$ be the (smooth) solution of the regularized closed-loop problem
\begin{equation}\label{eq:reg-problem-5KP-X}
\left\{
\begin{aligned}
\partial_t v^n
- \beta\,\partial_x^5 v^n
+ \gamma\,\partial_x^{-1}\partial_y^2 v^n
+ G D_x^5 G v^n
+ \varepsilon_n D_x^5 v^n &= F^n,\\
v^n(0) &= v_0^n .
\end{aligned}
\right.
\end{equation}
Applying Proposition \ref{prop:propagation} together with the basic
$L^\infty_tX_s$ energy estimate yields
$$
\|v^n\|_{Z^{s,T}}
\le C(s,T)\Bigl(\|v_0^n\|_{X_s}
+ \|F^n\|_{L^2(0,T;X_{s-5/2})}\Bigr),
$$
with $C(s,T)$ independent of $n$. Using \eqref{eq:BS-smoothing-est-5KP-X} and \eqref{eq:Fn-approx-X} we obtain the uniform bound
\begin{equation}\label{eq:uniform-Z-bound-5KP-X-data}
\|v^n\|_{Z^{s,T}}
\le C(s,T)\Bigl(\|v_0\|_{X_s}
+ \|F\|_{L^2(0,T;X_{s-5/2})}\Bigr),
\qquad \forall n \in\mathbb{N}.
\end{equation}
	
\medskip\noindent
	\textbf{ii. Bona–Smith argument in $Z^{s,T}$.} 
	\medskip
	
	For $m\ge n$, set $w^{n,m}:=v^n-v^m$. Subtracting \eqref{eq:reg-problem-5KP-X} for $n$ and $m$ gives
$$
\partial_t w^{n,m}
- \beta\,\partial_x^5 w^{n,m}
+ \gamma\,\partial_x^{-1}\partial_y^2 w^{n,m}
+ G D_x^5 G w^{n,m}
+ \varepsilon_m D_x^5 w^{n,m}
= (F^n-F^m) + (\varepsilon_m-\varepsilon_n)D_x^5 v^n,
$$
with initial data
$$
w^{n,m}(0)=v_0^n-v_0^m .
$$
Viewing $\varepsilon_m D_x^5 w^{n,m}$ as part of the regularized operator and setting
\[
\widetilde{F}^{n,m}
:=F^n-F^m+(\varepsilon_m-\varepsilon_n)D_x^5 v^n ,
\]
Proposition \ref{prop:propagation} applied with $\varepsilon=\varepsilon_m$ yields
\begin{equation}\label{eq:diff-est-5KP-X}
\begin{split}
\|w^{n,m}\|_{Z^{s,T}}
\le& C(s,T)\Bigl(
\|v_0^n-v_0^m\|_{X_s}
+ \|\widetilde{F}^{n,m}\|_{L^2(0,T;X_{s-5/2})}
\Bigr)\\
\le& C(s,T)\Bigl(
\|v_0^n-v_0^m\|_{X_s}
+ \|F^n-F^m\|_{L^2(0,T;X_{s-5/2})}
\\&+ |\varepsilon_m-\varepsilon_n|
\|D_x^5 v^n\|_{L^2(0,T;X_{s-5/2})}
\Bigr).
\end{split}
\end{equation}
Since $D_x^5:X_{s+5/2}\to X_{s-5/2}$ continuously,
\begin{equation*}
\begin{split}
\|D_x^5 v^n\|_{L^2(0,T;X_{s-5/2})}\le& \|v^n\|_{L^2(0,T;X_{s+5/2})}\\
\le& \|v^n\|_{Z^{s,T}}\\
\le& C(s,T)\Bigl(\|v_0\|_{X_s}
+ \|F\|_{L^2(0,T;X_{s-5/2})}\Bigr),
\end{split}
\end{equation*}
thanks to \eqref{eq:uniform-Z-bound-5KP-X-data}. As $|\varepsilon_m-\varepsilon_n|\le\varepsilon_n$, for $m\ge n$, we have
$$
|\varepsilon_m-\varepsilon_n|
\|D_x^5 v^n\|_{L^2(0,T;X_{s-5/2})}
\le \varepsilon_n C(s,T)\Bigl(\|v_0\|_{X_s}
+ \|F\|_{L^2(0,T;X_{s-5/2})}\Bigr)\to0 .
$$
The previous estime together with $v_0^n\to v_0$ in $X_s$ and $F^n\to F$ in $L^2(0,T;X_{s-5/2})$, give that the right-hand side of \eqref{eq:diff-est-5KP-X} tends to $0$, as $m,n\to\infty$. Hence $\{v^n\}$ is Cauchy in $Z^{s,T}$, and there exists
$v\in Z^{s,T}$ such that
\begin{equation}\label{eq:vn-to-v-X}
v^n \to v, \qquad \text{strongly in } Z^{s,T}.
\end{equation}
	
\medskip\noindent
	\textbf{iii. Identification of the limit and exponential bounds.} 
	\medskip
	
	Passing to the limit in \eqref{eq:reg-problem-5KP-X}, using \eqref{eq:vn-to-v-X}, \eqref{eq:Fn-approx-X}, and the continuity of the spatial operators on $X_{s+5/2}$, shows that $v$ solves \eqref{eq:limit-closed-loop-5KP-X} with $v(0)=v_0$ in the distributional sense, hence in $Z^{s,T}$. The a priori estimate \eqref{eq:apriori-limit-5KP-X} follows from \eqref{eq:uniform-Z-bound-5KP-X-data} by letting $n\to\infty$ and using \eqref{eq:vn-to-v-X}. Uniqueness and continuous dependence on $(v_0,F)$ follow directly from \eqref{eq:apriori-limit-5KP-X}: the difference of two solutions satisfies the homogeneous equation with zero data and therefore vanishes.

Finally, assume $s\ge0$ and $F\equiv0$. For each $n$, let $\mathcal{S}_{\varepsilon_n}(t)$ be the $C_0$–semigroup on
$X_{s,0}(\Omega)$ generated by
\[
L_{\varepsilon_n}
:= -\beta\,\partial_x^5
+ \gamma\,\partial_x^{-1}\partial_y^2
+ G D_x^5 G
+ \varepsilon_n D_x^5 .
\]
By Proposition \ref{prop:exp-decay-5KP}, there exist $C=C(s)>0$ and $\lambda>0$, independent of $\varepsilon\in(0,1)$, such that
\begin{equation}\label{eq:decay-epsn-X}
\|\mathcal{S}_{\varepsilon_n}(t)v_0^n\|_{X_s}
\le C e^{-\lambda t}\|v_0^n\|_{X_s},
\qquad t\ge0 .
\end{equation}
Setting $v^n(t):=\mathcal{S}_{\varepsilon_n}(t)v_0^n$, the convergence \eqref{eq:vn-to-v-X} implies that for each $T>0$,
\begin{equation}\label{eq:strong-conv-semigroup-X}
\sup_{t\in[0,T]}\|v^n(t)-v(t)\|_{X_s}\to0 .
\end{equation}
Since $v_0^n\to v_0$ in $X_s$, passing to the limit $n\to\infty$ in
\eqref{eq:decay-epsn-X} and using \eqref{eq:strong-conv-semigroup-X}
yields
\[
\|\mathcal{S}(t)v_0\|_{X_s}
=\|v(t)\|_{X_s}
\le C e^{-\lambda t}\|v_0\|_{X_s},
\qquad t\ge0,
\]
which is \eqref{eq:decay-limit-5KP-X}. This completes the proof.
\end{proof}

\section{Nonlinear theory: Well-posedness}\label{sec:NWP}

This section develops the nonlinear framework for the well-posedness of the closed-loop KP5 equation \eqref{eq:nonlin-closed-loop-again}. We prove bilinear estimates for $u,\partial_x u$ in the KP-adapted Sobolev scale, combine them with linear semigroup estimates, and apply a contraction argument in an exponentially weighted space to obtain small-data well-posedness in $X_{s,0}(\Omega)$, $s>2$, with data-independent existence time. This result is a key step toward the stabilization and control results that follow.

\subsection{Slab spaces and bilinear estimates}

Before stating the estimates, we introduce a piece of notation used throughout this subsection. For any $f\in L^1(\mathbb{T}_x;L^2(\mathbb{R}_y))$ we denote by
\[
\Pi_0 f
:= \frac{1}{2\pi}\int_{\mathbb{T}}f(x,\cdot)\,dx
\;\in L^2(\mathbb{R}_y)
\]
the \emph{$x$-mean projection} of $f$, i.e.\ the Fourier multiplier that retains only the $k=0$ mode in $x$. For any $\sigma\in\mathbb{R}$, $\Pi_0$ is bounded on $\mathcal{H}^\sigma$ and on $X_\sigma$, and
$$
(\mathrm{Id}-\Pi_0)f
= \sum_{k\neq 0}\widehat{f}(k,\cdot)\,e^{ikx}
$$
is the mean-zero part of $f$. The operator $\partial_x^{-1}$ acts on mean-zero functions via \eqref{antDer}, and for a general $f$ one has
$$
	\partial_x^{-1}\partial_x f
	= (\mathrm{Id}-\Pi_0)f
	= f-\Pi_0 f,
	$$
since $\partial_x^{-1}\partial_x$ is the identity on the mean-zero subspace and annihilates constants in $x$. The first result of this subsection gives us the bilinear estimates and can be read as follows. 
\begin{lemma}\label{lem:bilinear-kp5-X}
	Let $s>2$ and $T>0$. There exists $C=C(s,T)>0$ such that for all
	$u\in Z^{s,T}$,
	\begin{equation}\label{eq:bilinear-kp5-X}
		\|u\,\partial_x u\|_{L^2(0,T;\,X_{s-5/2})}
		\le C\,\|u\|_{Z^{s,T}}^2.
	\end{equation}
\end{lemma}

\begin{proof}
Recall that 
\[
\|f\|_{X_{s-5/2}}^2
= \|f\|_{\mathcal{H}^{s-5/2}}^2
+ \|\partial_x^{-1}f\|_{\mathcal{H}^{s-5/2}}^2.
\]
We treat these two terms separately using standard tools. First, the one-dimensional Sobolev embedding on $\mathbb{T}$ with values in $L^2(\mathbb{R}_y)$: for any $\alpha>\tfrac{1}{2}$,
\begin{equation}\label{eq:embed-x-Linfty}
\|f\|_{L_x^\infty L_y^2}
\lesssim \|f\|_{\mathcal{H}^{\alpha}}.
\end{equation}
Second, the product estimate in anisotropic Sobolev spaces: for $\sigma>-\tfrac{1}{2}$,
\begin{equation}\label{eq:tame-H-X}
\|fg\|_{\mathcal{H}^{\sigma}}
\lesssim
\|f\|_{L_x^\infty L_y^2}\,\|g\|_{\mathcal{H}^{\sigma}}
+ \|g\|_{L_x^\infty L_y^2}\,\|f\|_{\mathcal{H}^{\sigma}}.
\end{equation}
Since $s>2$ implies $\sigma=s-\tfrac{5}{2}>-\tfrac{1}{2}$, \eqref{eq:tame-H-X} applies.

\medskip

\noindent\textbf{Estimate in $\mathcal{H}^{s-5/2}$.}
Applying \eqref{eq:tame-H-X} with $f=u$, $g=\partial_x u$, we obtain
\[
\|u\,\partial_x u\|_{\mathcal{H}^{s-5/2}}
\lesssim
\|u\|_{L_x^\infty L_y^2}\,
\|\partial_x u\|_{\mathcal{H}^{s-5/2}}
+ \|\partial_x u\|_{L_x^\infty L_y^2}\,
\|u\|_{\mathcal{H}^{s-5/2}}.
\]
By \eqref{eq:embed-x-Linfty} and $s>2$,
\[
\|u\|_{L_x^\infty L_y^2}\lesssim \|u\|_{\mathcal{H}^{s-3/2}},
\qquad
\|\partial_x u\|_{L_x^\infty L_y^2}
\lesssim \|u\|_{\mathcal{H}^{s-1/2}}.
\]
Using $\partial_x:\mathcal{H}^{s-3/2}\to\mathcal{H}^{s-5/2}$ and the embeddings $\mathcal{H}^s\hookrightarrow\mathcal{H}^\theta$,
\[
\|u\,\partial_x u\|_{\mathcal{H}^{s-5/2}}
\lesssim
\|u\|_{\mathcal{H}^{s-3/2}}^2
+ \|u\|_{\mathcal{H}^{s-1/2}}\|u\|_{\mathcal{H}^{s-5/2}}
\lesssim \|u\|_{\mathcal{H}^s}^2.
\]
Integrating in time,
\begin{equation}\label{eq:bilin-H-part}
\|u\,\partial_x u\|_{L^2(0,T;\mathcal{H}^{s-5/2})}
\le C(s,T)\,
\|u\|_{L^\infty(0,T;\mathcal{H}^s)}\,
\|u\|_{L^2(0,T;\mathcal{H}^s)}
\le C(s,T)\,\|u\|_{Z^{s,T}}^2.
\end{equation}

\medskip

\noindent\textbf{Estimate of the $\partial_x^{-1}$-term.}
Since $u$ has zero mean in $x$, $u\,\partial_x u=\tfrac{1}{2}\partial_x(u^2)$, hence
\begin{equation}\label{eq:dxinv-uux}
\partial_x^{-1}(u\,\partial_x u)
= \tfrac{1}{2}\bigl(u^2-\Pi_0(u^2)\bigr).
\end{equation}
As $\Pi_0$ is a contraction on $\mathcal{H}^{s-5/2}$, it suffices to bound $u^2$. By \eqref{eq:tame-H-X} and \eqref{eq:embed-x-Linfty},
\[
\|u^2\|_{\mathcal{H}^{s-5/2}}
\lesssim
\|u\|_{L_x^\infty L_y^2}\,
\|u\|_{\mathcal{H}^{s-5/2}}
\lesssim \|u\|_{\mathcal{H}^s}^2.
\]
Thus,
\begin{equation}\label{eq:bilin-dxinv-part}
\|\partial_x^{-1}(u\,\partial_x u)\|_{L^2(0,T;\mathcal{H}^{s-5/2})}
\le C(s,T)\,\|u\|_{Z^{s,T}}^2.
\end{equation}

\medskip

Combining \eqref{eq:bilin-H-part}-\eqref{eq:bilin-dxinv-part}, we conclude
\[
\|u\,\partial_x u\|_{L^2(0,T;X_{s-5/2})}
\le C(s,T)\,\|u\|_{Z^{s,T}}^2,
\]
which is \eqref{eq:bilinear-kp5-X}.
\end{proof}

Let us define the slab space. To do that that, for each $n\in\mathbb{N}_0$ consider the space given by
\begin{align}\label{eq:slab-space-def}
	Z^{s,[n,n+1]}
	&:= C\bigl([n,n+1]; X_{s,0}(\Omega)\bigr)
	\cap L^2\bigl(n,n+1; X_{s+5/2,0}(\Omega)\bigr),
	\end{align}
endowed with the norm
	\begin{align}\label{eq:slab-norm-def}
	\vertiii{u}_n
	&:= \|u\|_{L^\infty(n,n+1;X_s)}
	+ \|u\|_{L^2(n,n+1;X_{s+5/2})}.
\end{align}
Now, on each unit interval, we also need a localized difference estimate, given by the following result.

\begin{lemma} \label{lem:nonlinear-est}  Let $s>2$ and $k\in\mathbb{N}$.  For $u,v\in Z^{s,[k-1,k]}$,
	\begin{equation}\label{eq:nonlinear-diff-est}
		\|u\,\partial_x u - v\,\partial_x v\|_{L^2(k-1,k;\,X_{s-5/2})}
		\le C\bigl(\vertiii{u}_{k-1}+\vertiii{v}_{k-1}\bigr)
		\vertiii{u-v}_{k-1},
	\end{equation}
	for some constant $C>0$ depending only on $s$.
\end{lemma}
\begin{proof}
Factor the difference as
\begin{equation}\label{eq:diff-factor}
u\,\partial_x u - v\,\partial_x v
= (u-v)\,\partial_x u + v\,\partial_x(u-v).
\end{equation}
We estimate each term in $X_{s-5/2}$ pointwise in time and then integrate over $[k-1,k]$. First, thanks to the Kato–Ponce argument, as used in Lemma \ref{lem:bilinear-kp5-X}, we have that
\begin{align*}
\|(u-v)\,\partial_x u\|_{X_{s-5/2}}
&\lesssim \|u-v\|_{X_s}\,\|u\|_{X_s},
\end{align*}
and
\begin{align*}
\|v\,\partial_x(u-v)\|_{X_{s-5/2}}
&\lesssim \|v\|_{X_s}\,\|u-v\|_{X_s}.
\end{align*}
Hence
\begin{equation}\label{eq:pointwise-diff}
\|u\,\partial_x u - v\,\partial_x v\|_{X_{s-5/2}}
\lesssim
\bigl(\|u\|_{X_s}+\|v\|_{X_s}\bigr)\,\|u-v\|_{X_s}.
\end{equation}
Note that, the estimate for $\partial_x^{-1}$ follows identically using \eqref{eq:dxinv-uux} and \eqref{eq:diff-factor}.

Now, integrating \eqref{eq:pointwise-diff} over $[k-1,k]$ and using H\"older, yields that
\[
\|u\,\partial_x u - v\,\partial_x v\|_{L^2(k-1,k;X_{s-5/2})}
\lesssim
\bigl(\|u\|_{L^\infty(k-1,k;X_s)}+\|v\|_{L^\infty(k-1,k;X_s)}\bigr)
\|u-v\|_{L^2(k-1,k;X_s)}.
\]
By definition of $\vertiii{\cdot}_{k-1}$, \eqref{eq:nonlinear-diff-est} holds true.
\end{proof}

To exploit the uniform exponential decay of the linear semigroup (Proposition \ref{prop:BS-limit-5KP-X}), we work on unit time slabs with smoothing index $5/2$, consistent with the feedback $G D_x^5 G$. 

For each $n\in\mathbb{N}_0$, consider the slab space \eqref{eq:slab-space-def} endowed with the norm \eqref{eq:slab-norm-def}, and let ${\mathcal{S}(t)}{t\ge0}$ be the $C_0$-semigroup on $X_{s,0}(\Omega)$ associated with \eqref{eq:limit-closed-loop-5KP-X} for $F\equiv0$. By Proposition \ref{prop:BS-limit-5KP-X}, $\mathcal{S}(t)$ satisfies the decay estimate \eqref{eq:decay-limit-5KP-X}, while solutions to the inhomogeneous problem obey the a priori bound \eqref{eq:apriori-limit-5KP-X}. These properties yield the following slabwise estimates.

\begin{proposition}\label{prop:slab-est-kp5}
Let $s\ge0$ and let $\{\mathcal{S}(t)\}_{t\ge0}$ be the linear semigroup associated with \eqref{eq:limit-closed-loop-5KP-X} in the homogeneous case $F\equiv0$. Then there exist $\lambda>0$ and constants $\tilde c_0,\tilde c_1>0$, independent of $s$, $n$, and $t$, such that for all $n\in\mathbb{N}_0$ and $t\in[n,n+1]$ the following hold:
\begin{itemize}
    \item[(i)] (\textit{Homogeneous estimate}) For every $v_0\in X_{s,0}(\Omega)$,
    \begin{equation}\label{eq:slab-hom-kp5}
        \vertiii{\mathcal{S}(t)v_0}_n
        \le \tilde c_0\,e^{-\lambda n}\,\|v_0\|_{X_s}.
    \end{equation}
    
    \item[(ii)] (\textit{Inhomogeneous estimate}) For every $F\in L^2_{\mathrm{loc}}(\mathbb{R}_+;X_{s-5/2,0}(\Omega))$,
  	\begin{equation}\label{eq:slab-inhom-kp5}
		\vertiii{\int_0^t \mathcal{S}(t-t')F(t')\,dt'}_n
		\le \tilde c_1\biggl(
		\|F\|_{L^2(n,n+1;X_{s-5/2})}
		+\sum_{k=1}^{n}
		e^{-\lambda(n-k)}\|F\|_{L^2(k-1,k;X_{s-5/2})}
		\biggr).
		\end{equation}
\end{itemize}
\end{proposition}

\begin{proof}  Fix $n\ge0$ and $t\in[n,n+1]$. By the semigroup property,
	\[
	\mathcal{S}(t)v_0=\mathcal{S}(t-n)\mathcal{S}(n)v_0.
	\]
	The exponential decay \eqref{eq:decay-limit-5KP-X} gives
	\[
	\|\mathcal{S}(n)v_0\|_{X_s}
	\le C\,e^{-\lambda n}\|v_0\|_{X_s}.
	\]
	On the unit interval $[n,n+1]$, the function $v(\tau):=\mathcal{S}(\tau)v_0$ solves the homogeneous problem 	\eqref{eq:limit-closed-loop-5KP-X} with initial data  $v(n)=\mathcal{S}(n)v_0$. Applying the linear estimate \eqref{eq:apriori-limit-5KP-X} on $[n,n+1]$ yields
	\[
	\vertiii{v}_n
	\le C(s,1)\,\|v(n)\|_{X_s}
	\le C(s,1)\,C\,e^{-\lambda n}\|v_0\|_{X_s}.
	\]
	Setting $\tilde c_0:=C(s,1)\,C$, \eqref{eq:slab-hom-kp5} holds.
	
	\medskip To prove \eqref{eq:slab-inhom-kp5}, define the Duhamel integral
	\[
	w(t):=\int_0^t \mathcal{S}(t-t')F(t')\,dt'.
	\]
	Fix $n\ge0$ and $t\in[n,n+1]$. By the semigroup property,
	\[
	w(t)
	= \mathcal{S}(t-n)\,w(n)
	+ \int_n^t \mathcal{S}(t-t')F(t')\,dt'.
	\]
	On $[n,n+1]$, the function $w$ solves the inhomogeneous problem \eqref{eq:limit-closed-loop-5KP-X} with initial data $w(n)$ and forcing $F\,\mathbf{1}_{[n,n+1]}$. The linear estimate \eqref{eq:apriori-limit-5KP-X} on this unit interval gives
	\begin{equation}\label{eq:inhom-slab-intermediate}
		\vertiii{w}_n
		\le C(s,1)\Bigl(
		\|w(n)\|_{X_s}
		+ \|F\|_{L^2(n,n+1;X_{s-5/2})}
		\Bigr).
	\end{equation}
	To estimate $\|w(n)\|_{X_s}$, we split the Duhamel integral over unit slabs:
	\[
	w(n)
	= \int_0^n \mathcal{S}(n-t')F(t')\,dt'
	= \sum_{k=1}^n \int_{k-1}^k
	\mathcal{S}(n-t')F(t')\,dt'.
	\]
	For $t'\in[k-1,k]$ we have $n-t'\ge n-k$, so the exponential decay \eqref{eq:decay-limit-5KP-X} yields
	\[
	\|\mathcal{S}(n-t')F(t')\|_{X_s}
	\le C\,e^{-\lambda(n-k)}\|F(t')\|_{X_s}.
	\]
	Integrating over $[k-1,k]$ and summing in $k$,
	\begin{equation}\label{eq:wn-bound}
		\|w(n)\|_{X_s}
		\le C\sum_{k=1}^n
		e^{-\lambda(n-k)}\|F\|_{L^2(k-1,k;X_{s-5/2})}.
	\end{equation}
	Inserting \eqref{eq:wn-bound} into \eqref{eq:inhom-slab-intermediate} and absorbing the constants into $\tilde c_1$ gives \eqref{eq:slab-inhom-kp5}.
\end{proof}

\subsection{Nonlinear system: Global well-posedness}
We combine the slab estimates of Proposition \ref{prop:slab-est-kp5} with the bilinear bounds of Lemmas \ref{lem:bilinear-kp5-X} and \ref{lem:nonlinear-est} to obtain global well-posedness for the nonlinear closed-loop KP5 equation under a smallness condition on the initial data, independent of the time horizon. 

Let $\mathcal{S}(t)$ be the semigroup associated with \eqref{eq:limit-closed-loop-5KP-X} in the homogeneous case $F\equiv0$, as given by Proposition \ref{prop:BS-limit-5KP-X}. Then the nonlinear equation \eqref{eq:nonlin-closed-loop-again} admits the mild formulation
$$
	u(t)
	= \mathcal{S}(t)u_0
	- \int_0^t \mathcal{S}(t-\tau)
	\bigl(u\,\partial_x u\bigr)(\tau)\,d\tau,
$$
which motivates the nonlinear map
\begin{equation}\label{eq:Gamma-def}
	\Gamma(u)(t)
	:= \mathcal{S}(t)u_0
	- \int_0^t \mathcal{S}(t-\tau)
	\bigl(u\,\partial_x u\bigr)(\tau)\,d\tau.
\end{equation}
We introduce the global space
$$
	\mathcal{X}
	:= \Bigl\{
	u\in C(\mathbb{R}_+; X_{s,0}(\Omega))
	\cap L^2_{\mathrm{loc}}(\mathbb{R}_+; X_{s+5/2,0}(\Omega))
	:\;
	\|u\|_{\mathcal{X}}<\infty
	\Bigr\},
$$
endowed with the exponentially weighted norm
$$
	\|u\|_{\mathcal{X}}
	:= \sup_{n\ge0}\,e^{\lambda n}\vertiii{u}_n,
$$
where $\lambda>0$ is the decay rate appearing in \eqref{eq:slab-hom-kp5} and $\vertiii{\cdot}_n$ is the slab norm
\eqref{eq:slab-norm-def}. Now, we are in a position to give the main result of this section.

\begin{theorem}\label{thm:gwp-kp5-X} Let $s>2$. There exists $\rho>0$ such that for any $u_0\in X_{s,0}(\Omega)$ with $\|u_0\|_{X_s}\le \rho,$ the nonlinear closed-loop KP5 equation
	\begin{equation}\label{eq:nonlin-closed-loop-again}
	\begin{cases}
		\partial_t u
		+ \beta\,\partial_x^5 u
		+ u\,\partial_x u
		+ \gamma\,\partial_x^{-1}\partial_y^{2}u
		+ G D_x^{5}G u
		= 0,\\
           u(0)=u_0,
           \end{cases}
	\end{equation}
	admits a unique global solution
	\[
	u\in C\bigl([0,\infty);X_{s,0}(\Omega)\bigr)
	\cap L^2_{\mathrm{loc}}\bigl([0,\infty);
	X_{s+5/2,0}(\Omega)\bigr).
	\]
Moreover, the data-to-solution map $u_0\mapsto u$ is locally Lipschitz from $\{u_0\in X_{s,0}(\Omega):\|u_0\|_{X_s}\le\rho\}$ into $\mathcal{X}$.
\end{theorem}
\begin{proof}
We prove that the map $\Gamma$ defined in \eqref{eq:Gamma-def} is a contraction on a closed ball of the global space $\mathcal{X}$. The argument combines the slab estimates of Proposition \ref{prop:slab-est-kp5} with the bilinear bounds of Lemmas \ref{lem:bilinear-kp5-X} and \ref{lem:nonlinear-est}.
 
 \medskip
Let $B_R:=\{u\in\mathcal{X}:\|u\|_{\mathcal{X}}\le R\}$. For $u\in B_R$ and $t\in[n,n+1]$, write
\[
\Gamma(u)=\mathcal{S}(t)u_0-\int_0^t \mathcal{S}(t-\tau)(u\partial_x u)(\tau)\,d\tau=:I_1+I_2.
\]
By \eqref{eq:slab-hom-kp5},
\[
\vertiii{I_1}_n\le \tilde c_0 e^{-\lambda n}\|u_0\|_{X_s}.
\]
Applying \eqref{eq:slab-inhom-kp5} with $F=-u\partial_x u$ and Lemma \ref{lem:bilinear-kp5-X},
\[
\vertiii{I_2}_n\le \tilde c_1 C\Big(\vertiii{u}_n^2+\sum_{k=1}^n e^{-\lambda(n-k)}\vertiii{u}_{k-1}^2\Big).
\]
Using $\vertiii{u}_j\le e^{-\lambda j}R$ and summing the geometric series,
\[
\|\Gamma(u)\|_{\mathcal{X}}
\le \tilde c_0\|u_0\|_{X_s}+\tilde c_1 C_\lambda R^2.
\]
Thus, for $u,v\in B_R$,
\[
u\partial_x u-v\partial_x v=(u-v)\partial_x u+v\partial_x(u-v).
\]
Using Lemma \ref{lem:nonlinear-est} and \eqref{eq:slab-inhom-kp5},
\[
\|\Gamma(u)-\Gamma(v)\|_{\mathcal{X}}
\le 2\tilde c_1 C_\lambda R\,\|u-v\|_{\mathcal{X}}.
\]

Now, choose
\[
R=\frac{1}{4\tilde c_1 C_\lambda},\qquad
\rho=\frac{R}{2\tilde c_0}.
\]
If $\|u_0\|_{X_s}\le\rho$, then $\Gamma(B_R)\subset B_R$ and $\Gamma$ is a $\tfrac12$-contraction. By Banach’s theorem, there exists a unique $u\in\mathcal{X}$ solving \eqref{eq:nonlin-closed-loop-again}.

Finally, for data $u_0,v_0$ with $\|u_0\|_{X_s},\|v_0\|_{X_s}\le\rho$, the same estimates yield
\[
\|u-v\|_{\mathcal{X}}
\le \tilde c_0\|u_0-v_0\|_{X_s}
+\tfrac12\|u-v\|_{\mathcal{X}},
\]
hence
\[
\|u-v\|_{\mathcal{X}}
\le 2\tilde c_0\|u_0-v_0\|_{X_s},
\]
showing the result.
\end{proof}

\section{A unique continuation principle}\label{sec:UCP}
Recall the homogeneous (regularized) fifth-order KP equation on $\mathbb{T}_x\times\mathbb{R}_y$, restricted to functions with zero $x$–mean:
\begin{equation}\label{eq:lin5KP-free}
\partial_t v
+\beta\,\partial_x^5 v
+\gamma\,\partial_x^{-1}\partial_y^2 v
+\varepsilon D_x^5 v=0,
\qquad (x,y,t)\in\mathbb{T}\times\mathbb{R}\times(0,T),
\end{equation}
where $0\le\varepsilon<1$ and $\partial_x^{-1}$ is defined on mean-zero functions as in \eqref{antDer}. Let $g\in C^\infty(\mathbb{T})$ satisfy $g\ge0$ and $\dis\int_{\mathbb{T}}g=1$, and define the nonempty open set
\[
\omega:=\{x\in\mathbb{T}:\, g(x)>0\}.
\]

The unique continuation property states that any solution of \eqref{eq:lin5KP-free} vanishing on $\omega\times\mathbb{R}\times(0,T)$ is identically zero. This property is fundamental for establishing observability estimates and, in turn, for the analysis of stabilization and controllability. Our first result proves a one-dimensional unique continuation property, for each fixed $\eta\in\mathbb{R}$, for a suitable modification of \eqref{eq:lin5KP-free}.
\begin{lemma}\label{lem:ucp-1d-eta}  Let $\eta\in\mathbb{R}$, $T>0$, and $u\in L^2\!\bigl((0,T);L^2(\mathbb{T})\bigr)$  be a distributional solution of
	\begin{equation}\label{eq:1d-ucp-problem}
		\partial_t u
		+\beta\,\partial_x^5 u
		-\gamma\eta^2\,\partial_x^{-1}u=0
		\quad\text{in }\mathbb{T}\times(0,T),
		\qquad\text{with}\qquad
		\int_{\mathbb{T}}u(x,t)\,dx=0,
	\end{equation}
	satisfying $u(x,t)=0$ for a.e.\ $(x,t)\in\omega\times(0,T)$. Then $u\equiv 0$ a.e.\ on $\mathbb{T}\times(0,T)$.
\end{lemma}

\begin{proof}
	We combine the explicit evolution of the Fourier modes with the space-time vanishing of the solution on \(\omega\times(0,T)\).
	
\medskip
\noindent\emph{\textbf{Step 1.} Mode dynamics.}
\medskip

	Expanding \(u\) in Fourier series with respect to the periodic variable \(x\), we write
	\[
	u(x,t)=\sum_{k\in\mathbb{Z}}\widehat{u}_k(t)e^{ikx}.
	\]
	The zero-mean condition implies that \(\widehat{u}_0(t)=0\) for all \(t\in(0,T)\). Taking Fourier coefficients in \eqref{eq:1d-ucp-problem}, we obtain, for each \(k\neq 0\),
	\[
	\frac{d}{dt}\widehat{u}_k(t)+i\omega_k\,\widehat{u}_k(t)=0,
	\qquad
	\omega_k:=\beta k^5+\gamma\frac{\eta^2}{k}.
	\]
	Hence
	\begin{equation}\label{eq:mode-solution}
		\widehat{u}_k(t)=c_k e^{-i\omega_k t},
		\qquad k\neq 0,
	\end{equation}
	for suitable constants \(c_k\in\mathbb{C}\).
	
\medskip
\noindent\emph{\textbf{Step 2.} Frequency grouping from the space-time vanishing.}
\medskip

	Let \(\varphi\in L^2(\mathbb{T})\) be supported in \(\omega\). Since \(u=0\) on \(\omega\times(0,T)\), we have
	\[
	F_\varphi(t):=\int_{\mathbb{T}}u(x,t)\varphi(x)\,dx=0
	\qquad\text{for a.e. }t\in(0,T).
	\]
	Using \eqref{eq:mode-solution}, we may write
	\[
	F_\varphi(t)
	=
	\sum_{k\neq 0}c_k\,\widehat{\varphi}(-k)e^{-i\omega_k t}.
	\]
	Set \(a_k:=c_k\,\widehat{\varphi}(-k)\). Since \((c_k)_{k\neq 0}\in\ell^2\) and \((\widehat{\varphi}(k))_{k\in\mathbb{Z}}\in\ell^2\), the Cauchy-Schwarz inequality yields
	\[
	\sum_{k\neq 0}|a_k|<\infty.
	\]
	Therefore the series defining \(F_\varphi\) converges absolutely and uniformly on \(\mathbb{R}\), so \(F_\varphi\) is continuous. Since it vanishes almost everywhere on \((0,T)\), it follows that
	\begin{equation}\label{eq:Fphi-zero}
		F_\varphi(t)=0
		\qquad\text{for every }t\in[0,T].
	\end{equation}
	
	We now group the modes corresponding to the same temporal frequency. Let
	\[
	\mathcal{F}:=\{\omega_k:\ k\neq 0\},
	\qquad
	\Lambda_\alpha:=\{k\in\mathbb{Z}\setminus\{0\}:\ \omega_k=\alpha\},
	\quad \alpha\in\mathcal{F}.
	\]
	Since \(|\omega_k|\to\infty\) as \(|k|\to\infty\), each set \(\Lambda_\alpha\) is finite. Thus \eqref{eq:Fphi-zero} may be rewritten as
	\[
	F_\varphi(t)
	=
	\sum_{\alpha\in\mathcal{F}}
	\Bigl(\sum_{k\in\Lambda_\alpha}a_k\Bigr)e^{-i\alpha t}=0,
	\qquad\text{on }[0,T].
	\]
	This is an absolutely convergent exponential series. By the uniqueness theorem for Bohr-Fourier coefficients of almost periodic functions (equivalently, the only almost periodic function with all Fourier coefficients equal to zero is the zero function; see \cite[Ch. 1, Thm. 1.19 and the subsequent remark]{Corduneanu1989}), we conclude that
	\begin{equation}\label{eq:Bohr-zero}
		\sum_{k\in\Lambda_\alpha}c_k\,\widehat{\varphi}(-k)=0,
		\qquad\text{for every }\alpha\in\mathcal{F}.
	\end{equation}
	Since \eqref{eq:Bohr-zero} holds for every \(\varphi\in L^2(\mathbb{T})\) supported in \(\omega\), it follows that, for each \(\alpha\in\mathcal{F}\), the trigonometric polynomial
$$
		f_\alpha(x):=\sum_{k\in\Lambda_\alpha}c_k e^{ikx},
$$
	vanishes on \(\omega\).
	
\medskip
\noindent\emph{\textbf{Step 3.} Elimination of each frequency group.}
\medskip

	Fix \(\alpha\in\mathcal{F}\). Since \(\Lambda_\alpha\) is finite, \(f_\alpha\) is a trigonometric polynomial. As \(f_\alpha\) vanishes on the nonempty open set \(\omega\), we infer that $f_\alpha\equiv 0,$ on $\mathbb{T}$. Indeed, a nontrivial trigonometric polynomial cannot vanish on an open subset of the torus.
	
	We now decompose \(f_\alpha\) into its positive- and negative-frequency parts,
	\[
	f_\alpha^+(x)
	:=
	\sum_{\substack{k\in\Lambda_\alpha\\ k\ge 1}}c_k e^{ikx},
	\qquad
	f_\alpha^-(x)
	:=
	\sum_{\substack{k\in\Lambda_\alpha\\ k\le -1}}c_k e^{ikx}.
	\]
	Since \(f_\alpha=f_\alpha^++f_\alpha^-\equiv 0\) on \(\mathbb{T}\), and the two sums have disjoint Fourier support, the uniqueness of the Fourier expansion yields
$$
		f_\alpha^+\equiv 0
		\qquad\text{and}\qquad
		f_\alpha^-\equiv 0
		\qquad\text{on }\mathbb{T}.
$$
In particular, \(f_\alpha^+\) vanishes on the open interval \((a,b)\subset\omega\) and has a one-sided Fourier expansion supported on \(\{k\ge 1\}\). Thus, the hypotheses of \cite[Lemma 2.9]{LR} are satisfied, and we conclude that \(f_\alpha^+\equiv 0\), which means that \(c_k=0\) for every \(k\in\Lambda_\alpha\) with \(k\ge 1\).
	
	For the negative-frequency part, observe that
	\[
	\overline{f_\alpha^-(x)}
	=
	\sum_{\substack{k\in\Lambda_\alpha\\ k\le -1}}\overline{c_k}\,e^{-ikx}
	=
	\sum_{j\ge 1}\overline{c_{-j}}\,e^{ijx},
	\]
	which is again a one-sided Fourier series supported on positive integers. Since \(f_\alpha^-\equiv 0\) on \(\mathbb{T}\), the same is true for \(\overline{f_\alpha^-}\), and in particular \(\overline{f_\alpha^-}=0\) on \((a,b)\). Another application of \cite[Lemma 2.9]{LR} shows that \(\overline{c_{-j}}=0\) for all \(j\ge 1\), and hence \(c_k=0\) for every \(k\in\Lambda_\alpha\) with \(k\le -1\).
	
	Therefore \(c_k=0\) for all \(k\in\Lambda_\alpha\). Since \(\alpha\in\mathcal{F}\) was arbitrary and
	\[
	\bigcup_{\alpha\in\mathcal{F}}\Lambda_\alpha=\mathbb{Z}\setminus\{0\},
	\]
	we conclude that \(c_k=0\) for every \(k\neq 0\).
	
\medskip

Finally, recalling that \(\widehat{u}_0\equiv 0\) and that
	\(\widehat{u}_k(t)=c_k e^{-i\omega_k t}\) for \(k\neq 0\), we obtain
	\[
	\widehat{u}_k(t)=0
	\qquad\text{for all }k\in\mathbb{Z},\ t\in(0,T).
	\]
	Hence \(u\equiv 0\) on \(\mathbb{T}\times(0,T)\), which completes the proof.
\end{proof}

We are in a position to prove the unique continuation property for linear fifth-order KP, that is, proof of Theorem \ref{Th:UCP-5KP}.

\begin{proof}[Proof of Theorem \ref{Th:UCP-5KP}]
	We divide the argument into three steps. We first show that the trace \(c(t)\) must vanish identically, so that the solution itself vanishes on the control region. We then treat the regularized case separately \(\varepsilon>0\) and the purely dispersive case \(\varepsilon=0\).
	
\medskip
\noindent\emph{\textbf{Step 1.} Reduction to the case \(c\equiv 0\).}
\medskip

	For almost every \(t\in(0,T)\), we have \(v(\cdot,\cdot,t)\in L^2(\mathbb{T}\times\mathbb{R})\), and by \eqref{eq:UCP-assumption},
	\[
	v(x,y,t)=c(t)
	\qquad\text{for a.e. }(x,y)\in\omega\times\mathbb{R}.
	\]
	More precisely, this conclusion holds for every \(t\) in a set of full measure for which \(v(\cdot,\cdot,t)\in L^2(\mathbb{T}\times\mathbb{R})\), the existence of such a set being guaranteed by Fubini's theorem and the assumption \(v\in L^2((0,T);L^2(\mathbb{T}\times\mathbb{R}))\). Therefore,
	\[
	\|v(\cdot,\cdot,t)\|_{L^2(\omega\times\mathbb{R})}^2
	=
	\int_\omega\int_{\mathbb{R}} |v(x,y,t)|^2\,dy\,dx
	=
	|\omega|\,|c(t)|^2 \int_{\mathbb{R}}1\,dy.
	\]
	Since the left-hand side is finite and \(\dis\int_{\mathbb{R}}1\,dy=+\infty\),
	it follows that \(c(t)=0\) for almost every \(t\in(0,T)\). Hence
	\[
	v=0
	\qquad\text{a.e. on }\omega\times\mathbb{R}\times(0,T).
	\]
	It remains to prove that this implies \(v\equiv 0\) in
	\(\mathbb{T}\times\mathbb{R}\times(0,T)\).
	
\medskip
\noindent\emph{\textbf{Step 2.} The case \(\varepsilon>0\).}
\medskip

	We expand \(v\) in Fourier series with respect to the periodic variable \(x\):
	\[
	v(x,y,t)=\sum_{k\neq 0} v_k(y,t)e^{ikx},
	\]
	where the zero mode is absent by the \(x\)-mean-zero assumption. For each \(k\neq 0\), the corresponding Fourier coefficient satisfies
	\[
	\partial_t v_k
	+\bigl(i\beta k^5+\varepsilon |k|^5\bigr)v_k
	-\frac{i\gamma}{k}\,\partial_y^2 v_k=0
	\qquad\text{in }\mathbb{R}_y\times(0,T).
	\]
	Since \(v\in L^2((0,T);L^2(\mathbb{T}\times\mathbb{R}))\), there exists a set \(E\subset(0,T)\) of full measure such that
	\[
	v(\cdot,\cdot,t_0)\in L^2(\mathbb{T}\times\mathbb{R})
	\qquad\text{for every }t_0\in E.
	\]
	Fix \(t_0\in E\). For each \(k\neq 0\), the above equation generates a linear evolution on \(L^2(\mathbb{R}_y)\), and the standard \(L^2\)-energy identity yields, for every \(t\in(t_0,T)\),
	\begin{equation}\label{eq:mode-decay-eps-rev}
		\|v_k(\cdot,t)\|_{L^2_y}
		=
		e^{-\varepsilon |k|^5(t-t_0)}
		\|v_k(\cdot,t_0)\|_{L^2_y}.
	\end{equation}
	Indeed, the term \(i\beta k^5 v_k\) is purely skew-adjoint, and so is \(-\frac{i\gamma}{k}\partial_y^2 v_k\), while the dissipative contribution \(\varepsilon |k|^5 v_k\) produces the exponential decay.
	
	Fix now \(t\in(t_0,T)\). From \eqref{eq:mode-decay-eps-rev}, for every \(z\in\mathbb{C}\) we obtain
	\[
	\sum_{k\neq 0}\|v_k(\cdot,t)\|_{L^2_y}\,|e^{ikz}|
	\le
	\sum_{k\neq 0}
	e^{-\varepsilon |k|^5(t-t_0)}\,
	\|v_k(\cdot,t_0)\|_{L^2_y}\,
	e^{|k||\mathrm{Im} z|}.
	\]
	By Cauchy-Schwarz,
	\[
	\sum_{k\neq 0}
	e^{-\varepsilon |k|^5(t-t_0)}
	\|v_k(\cdot,t_0)\|_{L^2_y}\,
	e^{|k||\mathrm{Im} z|}
	\le
	\Bigl(\sum_{k\neq 0}\|v_k(\cdot,t_0)\|_{L^2_y}^2\Bigr)^{1/2}
	\Bigl(\sum_{k\neq 0}e^{-2\varepsilon |k|^5(t-t_0)+2|k||\mathrm{Im} z|}\Bigr)^{1/2},
	\]
	and the second series is finite because the factor \(e^{-2\varepsilon |k|^5(t-t_0)}\) dominates any exponential growth in \(|k|\). Hence, for each fixed \(t\in(t_0,T)\), the series
	\[
	\sum_{k\neq 0} v_k(\cdot,t)e^{ikz}
	\]
	converges absolutely in \(L^2(\mathbb{R}_y)\) for every \(z\in\mathbb{C}\). It therefore defines an \(L^2(\mathbb{R}_y)\)-valued entire periodic function of \(z\), and in particular the map
	\[
	x\longmapsto v(x,\cdot,t)
	\]
	is \(L^2(\mathbb{R}_y)\)-valued real-analytic on \(\mathbb{T}\).
	
	On the other hand, since \(v=0\) almost everywhere on \(\omega\times\mathbb{R}\times(0,T)\), Fubini's theorem implies that for almost every \(x\in\omega\),
	\[
	v(x,\cdot,t)=0
	\qquad\text{in }L^2(\mathbb{R}_y)
	\]
	for almost every \(t\in(0,T)\). In particular, for the fixed \(t\in(t_0,T)\) under consideration, the \(L^2(\mathbb{R}_y)\)-valued analytic function \(x\mapsto v(x,\cdot,t)\) vanishes on a subset of \(\omega\) of positive measure, hence on a set with an accumulation point. By the identity theorem for Banach-valued holomorphic functions, it follows that
	\[
	v(x,\cdot,t)\equiv 0
	\qquad\text{for every }x\in\mathbb{T}.
	\]
	Thus \(v(\cdot,\cdot,t)\equiv 0\) for every \(t\in(t_0,T)\).
	
	Finally, let \(t\in(0,T)\) be arbitrary. Since \(E\) has full measure, we may choose \(t_0\in E\cap(0,t)\). The above argument then gives 	\(v(\cdot,\cdot,t)\equiv 0\). Hence
	\[
	v\equiv 0
	\qquad\text{a.e. in }\mathbb{T}\times\mathbb{R}\times(0,T)
	\]
	when \(\varepsilon>0\).
	
\medskip
\noindent\emph{\textbf{Step 3.} The case \(\varepsilon=0\).}
\medskip

	We now take the Fourier transform in the transverse variable \(y\):
	\[
	\widehat v(x,\eta,t)
	=
	\int_{\mathbb{R}} v(x,y,t)e^{-i\eta y}\,dy.
	\]
	Since \(\partial_y^2 v\) is transformed into \(-\eta^2 \widehat v\), we obtain, for almost every \(\eta\in\mathbb{R}\), that
	\[
	u(x,t):=\widehat v(x,\eta,t)
	\]
	satisfies
	\[
	\partial_t u+\beta\,\partial_x^5 u-\gamma \eta^2 \partial_x^{-1}u=0
	\qquad\text{in }\mathbb{T}\times(0,T).
	\]
	Moreover, by Plancherel's theorem,
	\[
	u\in L^2\bigl((0,T);L^2(\mathbb{T})\bigr)
	\qquad\text{for a.e. }\eta\in\mathbb{R}.
	\]
	
	The \(x\)-mean-zero property is preserved under the Fourier transform in \(y\). Indeed, for almost every \(\eta\in\mathbb{R}\),
	\[
	\int_{\mathbb{T}}u(x,t)\,dx
	=
	\int_{\mathbb{T}}\widehat v(x,\eta,t)\,dx
	=
	\widehat{\Bigl(\int_{\mathbb{T}} v(x,\cdot,t)\,dx\Bigr)}(\eta)
	=
	0
	\qquad\text{for a.e. }t\in(0,T).
	\]
	It remains to verify the vanishing in the control region. Since \(v=0\) almost everywhere on \(\omega\times\mathbb{R}\times(0,T)\), we have, for almost every \((x,t)\in\omega\times(0,T)\),
	\[
	v(x,\cdot,t)=0
	\qquad\text{in }L^2(\mathbb{R}_y).
	\]
	Taking the Fourier transform in \(y\), it follows that
	\[
	\widehat v(x,\eta,t)=0
	\qquad\text{for a.e. }\eta\in\mathbb{R}.
	\]
	By Fubini's theorem, for almost every \(\eta\in\mathbb{R}\) we therefore obtain
	\[
	u(x,t)=\widehat v(x,\eta,t)=0
	\qquad\text{for a.e. }(x,t)\in\omega\times(0,T).
	\]
	
	Thus, for almost every \(\eta\in\mathbb{R}\), the function \(u=\widehat v(\cdot,\eta,\cdot)\) satisfies all the hypotheses of Lemma \ref{lem:ucp-1d-eta}. We conclude that
	\[
	u\equiv 0
	\qquad\text{on }\mathbb{T}\times(0,T)
	\]
	for almost every \(\eta\in\mathbb{R}\). By the injectivity of the Fourier transform on \(L^2(\mathbb{R}_y)\), this implies
	\[
	v\equiv 0
	\qquad\text{a.e. in }\mathbb{T}\times\mathbb{R}\times(0,T).
	\]
	This completes the proof.
\end{proof}

\section{Controllability results}\label{sec:ucp-obs}

In this section, we establish an observability inequality that is crucial for proving uniform exponential decay of the associated closed-loop semigroup on the KP-adapted Sobolev space $X_s$, as well as for controllability. Fix $0<\varepsilon<1$, and let $\{\mathcal{S}_\varepsilon(t)\}_{t\ge0}$ be the analytic semigroup on $X_{0,0}(\Omega)$ generated by $-\mathcal{L}_\varepsilon$ (see Proposition \ref{prop:semigroup-5KP}), where
\[
\mathcal{L}_\varepsilon
:= \varepsilon D_x^5-\beta\,\partial_x^5
+\gamma\,\partial_x^{-1}\partial_y^2
+GD_x^5G,
\qquad
\mathcal{D}(\mathcal{L}_\varepsilon)=X_{5,0}(\Omega).
\]
We also denote by $\{\mathcal{S}(t)\}_{t\ge0}$ the limiting $C_0$-semigroup on $X_{s,0}(\Omega)$ obtained as $\varepsilon\to0$ (Proposition \ref{prop:BS-limit-5KP-X}), and by $\{\mathcal{S}_\varepsilon^*(t)\}_{t\ge0}$ the adjoint semigroup, associated with the adjoint operator given by
\[
\mathcal{L}_\varepsilon^*
=\varepsilon D_x^5+\beta\,\partial_x^5
-\gamma\,\partial_x^{-1}\partial_y^2
+GD_x^5G,
\]
so that $\mathcal{L}_\varepsilon^*\neq\mathcal{L}_\varepsilon$ in general. Nevertheless, $\mathcal{S}_\varepsilon(t)$, $\mathcal{S}_\varepsilon^*(t)$, and $\mathcal{S}(t)$ all satisfy uniform exponential decay estimates with the same constants $C,\lambda>0$ (Propositions \ref{prop:exp-decay-5KP}, \ref{prop:adjoint-decay}, and \ref{prop:limit-semigroup-decay}), since $\mathcal{L}_\varepsilon$ and $\mathcal{L}_\varepsilon^*$ share the same dissipative structure.

\subsection{Linear stabilization} Thanks to the unique continuation property, established in Theorem \ref{Th:UCP-5KP}, we derive an observability inequality for the regularized flow and, as a consequence, a decay estimate in $X_s$ which is uniform with respect to the regularization parameter. The next result in this section ensures an exponential decay of the solution uniformly in $\varepsilon$, and can be read as follows.
\begin{proposition}\label{prop:exp-decay-5KP}  	Let $0<\varepsilon<1$ and $s\ge0$. There exist constants $C,\lambda>0$, independent of  $\varepsilon$, such that
	\begin{equation}\label{eq:exp-decay-5KP-Xs}
		\|\mathcal{S}_\varepsilon(t)v_0\|_{X_s}
		\le C e^{-\lambda t}\|v_0\|_{X_s},
		\qquad \forall t\ge0,\quad v_0\in X_{s,0}(\Omega).
	\end{equation}
\end{proposition}
\begin{proof} The proof is divided into two parts. First, we establish exponential decay in the base space $X_0$ through an observability inequality derived from Theorem \ref{Th:UCP-5KP}. Next, we extend this decay to $X_s$, for $s\ge0$, by exploiting the smoothing properties of the analytic semigroup.
	
\medskip\noindent
	\emph{\textbf{Part 1.} Decay in $X_0$ via observability.}
\vspace{0.2cm}

Fix $T>0$ and define $v(t)=\mathcal{S}_\varepsilon(t)v_0$ for $v_0\in X_{0,0}(\Omega)$. Then $v$ solves the homogeneous closed-loop system
\begin{equation}\label{eq:closed-loop-hom-5KP-X}
\partial_t v
+ \beta\,\partial_x^5 v
+ \gamma\,\partial_x^{-1}\partial_y^2 v
+ \varepsilon D_x^5 v
+ G D_x^5 G v
= 0
\quad\text{in }\Omega\times(0,T).
\end{equation}
Taking the $L^2(\Omega)$ inner product of \eqref{eq:closed-loop-hom-5KP-X} with $v(t)$, and using that $\partial_x^5$ is skew-adjoint on the mean-zero subspace, as well as the skew-adjointness of $\partial_x^{-1}\partial_y^2$ on the KP domain, we obtain
$$
\frac{d}{dt}\|v(t)\|_{L^2}^2
+ 2\varepsilon\|D_x^{5/2}v(t)\|_{L^2}^2
+ 2\|D_x^{5/2}(Gv(t))\|_{L^2}^2
= 0,
$$
since the operators $G D_x^5 G$ and $\varepsilon D_x^5$ are self-adjoint and nonnegative. Integrating over $[0,T]$ yields
$$
\|v(T)\|_{L^2}^2
+ 2\varepsilon\int_0^T \|D_x^{5/2}v(t)\|_{L^2}^2\,dt
+ 2\int_0^T \|D_x^{5/2}(Gv(t))\|_{L^2}^2\,dt
= \|v_0\|_{L^2}^2.
$$

We now derive a companion energy identity for
$w(t):=\partial_x^{-1}v(t)$.  Applying $\partial_x^{-1}$ to
\eqref{eq:closed-loop-hom-5KP-X} and using that $\partial_x^{-1}$
commutes with $\partial_x^5$, $D_x^5$, and
$\partial_x^{-1}\partial_y^2$ (all Fourier multipliers), we obtain
\begin{equation}\label{eq:w-eq-obs}
	\partial_t w
	+ \beta\,\partial_x^5 w
	+ \gamma\,\partial_x^{-1}\partial_y^2 w
	+ \varepsilon D_x^5 w
	+ G D_x^5 G w
	= -Rv,
\end{equation}
where the commutator remainder is
\begin{equation}\label{eq:R-def-obs}
	R:=[\partial_x^{-1},\,G D_x^5 G]
	= [\partial_x^{-1},G]\,D_x^5\,G
	+ G\,D_x^5\,[\partial_x^{-1},G].
\end{equation}
Since $[\partial_x^{-1},G]$ has order $-2$ in $x$ (see
\eqref{eq:comm-dxinv-G-bound}), the operator $R$ has order at
most $3$ and satisfies
$$
	\|Rf\|_{L^2}
	\le C(g)\,\|f\|_{\mathcal{H}^{3}},
	\qquad f\in\mathcal{H}_0^{3}(\Omega).
$$

Taking the $L^2(\Omega)$ inner product of \eqref{eq:w-eq-obs}
with $w$ and arguing exactly as for $v$, we obtain
\begin{equation}\label{eq:energy-dxinv-X}
	\frac12\frac{d}{dt}\|w(t)\|_{L^2}^2
	+ \varepsilon\|D_x^{5/2}w(t)\|_{L^2}^2
	+ \|D_x^{5/2}(Gw(t))\|_{L^2}^2
	= -(Rv(t),w(t))_{L^2}.
\end{equation}
Summing the $L^2$ identity for $v$ (which is exact, since no
commutator arises) with \eqref{eq:energy-dxinv-X}, and integrating
over $[0,T]$, we arrive at the \emph{modified} $X_0$ energy identity
\begin{equation}\label{eq:energy-identity-X0}
	\begin{split}
		\|v(T)\|_{X_0}^2
		&+ 2\varepsilon\int_0^T
		\bigl(\|D_x^{5/2}v\|_{L^2}^2+\|D_x^{5/2}w\|_{L^2}^2\bigr)\,dt\\
		&+ 2\int_0^T
		\bigl(\|D_x^{5/2}(Gv)\|_{L^2}^2+\|D_x^{5/2}(Gw)\|_{L^2}^2\bigr)\,dt
		= \|v_0\|_{X_0}^2
		- 2\int_0^T (Rv,w)_{L^2}\,dt.
	\end{split}
\end{equation}

\noindent{\textbf{Claim.}} \textit{There exists $C_T>0$, independent of $0<\varepsilon<1$, such that
	\begin{equation}\label{eq:observability-5KP-X0}
		\|v_0\|_{X_0}^2
		\le C_T\left(
		\varepsilon\int_0^T \|D_x^{5/2}v(t)\|_{X_0}^2\,dt
		+ \int_0^T \|D_x^{5/2}(Gv(t))\|_{X_0}^2\,dt
		\right).
	\end{equation}}
	
	Assume, by contradiction, that \eqref{eq:observability-5KP-X0} fails for some fixed $T>0$. Then there exist sequences $\varepsilon_n\in(0,1)$ and $v_0^n\in X_{0,0}(\Omega)$ such that, defining $v^n(t):=\mathcal{S}_{\varepsilon_n}(t)v_0^n$ and normalizing $\|v_0^n\|_{X_0}=1$, one has
$$
1
> n\left(
\varepsilon_n\int_0^T \|D_x^{5/2}v^n(t)\|_{X_0}^2\,dt
+ \int_0^T \|D_x^{5/2}(Gv^n(t))\|_{X_0}^2\,dt
\right).
$$
In particular,
\begin{equation}\label{eq:dissipation-vanish-X0}
\varepsilon_n\int_0^T \|D_x^{5/2}v^n\|_{X_0}^2\,dt \to 0,
\qquad
\int_0^T \|D_x^{5/2}(Gv^n)\|_{X_0}^2\,dt \to 0.
\end{equation}
By Proposition \ref{prop:uniform-Zs} with $s=0$ and $F\equiv0$, the sequence $\{v^n\}$ is bounded in
$Z^{0,T}= L^\infty(0,T;X_0)\cap L^2(0,T;X_{5/2})$. Moreover, from
\eqref{eq:closed-loop-hom-5KP-X} we have
\begin{equation}\label{eq:vn-time-derivative}
\partial_t v^n
= -\beta\,\partial_x^5 v^n
- \gamma\,\partial_x^{-1}\partial_y^2 v^n
- \varepsilon_n D_x^5 v^n
- G D_x^5 G v^n.
\end{equation}
Using the mapping properties of the operators together with the bound of $v^n$ in $L^2(0,T;X_{5/2})$, the right-hand side of \eqref{eq:vn-time-derivative} is uniformly bounded in $L^2(0,T;X_{-5/2})$. Indeed,
\[
\partial_x^5,\ \partial_x^{-1}\partial_y^2,\ D_x^5,\ G D_x^5 G
: X_{5/2}\to X_{-5/2}
\]
continuously, and
\[
\|G D_x^5 G f\|_{X_{-5/2}}
\lesssim \|D_x^{5/2}(Gf)\|_{X_0},
\]
by Lemma \ref{lem:G-properties} and duality. Hence, the Aubin-Lions-Simon compactness lemma\footnote{See, for instance, \cite{{Aubin,Simon,Lions1988}}} applied to
\[
X_{5/2}\hookrightarrow\hookrightarrow X_0\hookrightarrow X_{-5/2}
\]
yields, up to a subsequence, the existence of $v\in Z^{0,T}$ such that
\begin{equation}\label{eq:compactness-5KP-X}
v^n\to v \quad\text{strongly in }L^2(0,T;X_0),
\qquad
v^n\rightharpoonup v \quad\text{weakly in }L^2(0,T;X_{5/2}).
\end{equation}
We also record that, by the modified energy identity
\eqref{eq:energy-identity-X0},
\begin{equation}\label{eq:energy-identity-X0-n}
	\|v^n(t)\|_{X_0}^2
	= 1
	- 2\varepsilon_n\!\int_0^t\!\bigl(\|D_x^{5/2}v^n\|_{L^2}^2
	+\|D_x^{5/2}w^n\|_{L^2}^2\bigr)\,d\tau
	- 2\!\int_0^t\!\mathcal{D}_G^n(\tau)\,d\tau
	+ 2\!\int_0^t\!(Rv^n,w^n)\,d\tau,
\end{equation}
where
$w^n:=\partial_x^{-1}v^n$,
$\mathcal{D}_G^n
:=\|D_x^{5/2}(Gv^n)\|_{L^2}^2+\|D_x^{5/2}(Gw^n)\|_{L^2}^2$,
and $R=[\partial_x^{-1},GD_x^5G]$ is the commutator defined in
\eqref{eq:R-def-obs}.  The first two dissipation integrals tend
to $0$ by \eqref{eq:dissipation-vanish-X0}.

Up to the extraction of a further subsequence, we may assume $\varepsilon_n\to\varepsilon_\ast\in[0,1]$. From \eqref{eq:dissipation-vanish-X0}, we have
\[
D_x^{5/2}(Gv^n)\to0 \quad\text{in }L^2(0,T;X_0).
\]
By Poincaré’s inequality in the $x$–variable on the mean-zero subspace,
\[
\|Gv^n\|_{X_0}
\lesssim \|D_x^{5/2}(Gv^n)\|_{X_0}\to0,
\]
and therefore
$$
Gv^n\to0 \quad\text{in }L^2(0,T;X_0),
\qquad\text{hence}\qquad
Gv=0 \ \text{a.e. in }\Omega\times(0,T).
$$
Moreover,
\[
\|G D_x^5 G v^n\|_{L^2(0,T;X_{-5/2})}
\lesssim \|D_x^{5/2}(Gv^n)\|_{L^2(0,T;X_0)}\to0.
\]
So the feedback term vanishes in the limit.

Passing to the limit in \eqref{eq:vn-time-derivative} in the sense of distributions, using \eqref{eq:compactness-5KP-X} and the previous convergence, we deduce that $v$ solves
$$
\partial_t v
+ \beta\,\partial_x^5 v
+ \gamma\,\partial_x^{-1}\partial_y^2 v
+ \varepsilon_\ast D_x^5 v
= 0
\quad\text{in }\Omega\times(0,T),
$$
with zero $x$–mean and satisfying $Gv=0$ a.e. By the definition of $G$ (see \eqref{eq:G-def-prelim}), this implies that for a.e.\ $(y,t)$,
\[
g(x)\Bigl(v(x,y,t)-\int_{\mathbb{T}}g(\xi)v(\xi,y,t)\,d\xi\Bigr)=0
\quad\text{for all }x\in\mathbb{T}.
\]
Since $g(x)>0$ on $\omega$, it follows that $v$ is independent of $x$ on $\omega$, and hence
\[
\partial_x v(x,y,t)=0 \quad\text{for a.e. }(x,y,t)\in\omega\times\mathbb{R}\times(0,T).
\]
Setting $w:=\partial_x v$, we see that $w$ has zero $x$–mean and satisfies
$$
\partial_t w
+ \beta\,\partial_x^5 w
+ \gamma\,\partial_x^{-1}\partial_y^2 w
+ \varepsilon_\ast D_x^5 w
= 0,
$$
with $w=0$ on $\omega\times\mathbb{R}\times(0,T)$. By Theorem \ref{Th:UCP-5KP}, we conclude that $w\equiv0$, and since $v$ has zero mean in $x$, it follows that $v\equiv0$ in $\Omega\times(0,T)$.

It remains to show that the commutator integral
$\dis\int_0^T(Rv^n,w^n)\,dt$ tends to zero, where
$w^n:=\partial_x^{-1}v^n$.  Recall from \eqref{eq:R-def-obs} that
\[
R
= [\partial_x^{-1},G]\,D_x^5\,G
+ G\,D_x^5\,[\partial_x^{-1},G],
\]
and that $[\partial_x^{-1},G]$ has order $-2$ in $x$
(see \eqref{eq:comm-dxinv-G-bound}).  We treat the two summands
separately.  Note that the argument below uses the fact, already
established above via the unique continuation property, that $v\equiv0$
and hence $v^n\to 0$ strongly in $L^2(0,T;X_0)$.

\medskip\noindent
\emph{First summand.} Since $\|D_x^{5/2}(Gv^n)\|_{L^2(0,T;L^2)}\to 0$ by \eqref{eq:dissipation-vanish-X0} and $Gv^n$ is mean-zero, Poincar\'e's inequality in the $x$-variable gives $Gv^n\to 0$ in $L^2(0,T;\mathcal{H}^{5/2})$, hence $D_x^5(Gv^n)\to 0$ in $L^2(0,T;\mathcal{H}^{-5/2})$. As $[\partial_x^{-1},G]$ has order $-2$,
\[
[\partial_x^{-1},G]\,D_x^5(Gv^n)\to 0
\qquad\text{in }L^2(0,T;\mathcal{H}^{-1/2}).
\]
Because $\{w^n\}$ is bounded in $L^2(0,T;\mathcal{H}^{5/2}) \subset L^2(0,T;\mathcal{H}^{1/2})$ by
Proposition \ref{prop:uniform-Zs}, the $\mathcal{H}^{-1/2}$-$\mathcal{H}^{1/2}$ duality pairing yields
\begin{equation}\label{eq:first-summand-to-0}
	\int_0^T\bigl([\partial_x^{-1},G]\,D_x^5 Gv^n,\,w^n\bigr)\,dt
	\to 0.
\end{equation}

\medskip\noindent
\emph{Second summand.} We use the self-adjointness of $G$ on $L^2(\Omega)$ together with the identity
\[
(D_x^5 f,g)_{L^2}=(D_x^{5/2}f,D_x^{5/2}g)_{L^2},
\]
which holds on the $x$-mean-zero subspace. Therefore,
\begin{equation}\label{eq:second-summand-refactor}
	\int_0^T\bigl(G\,D_x^5\,[\partial_x^{-1},G]\,v^n,\;w^n\bigr)\,dt
	=
	\int_0^T\bigl(D_x^{5/2}\,[\partial_x^{-1},G]\,v^n,\;
	D_x^{5/2}(G\,w^n)\bigr)\,dt.
\end{equation}
We estimate the two factors on the right-hand side separately.

\smallskip\noindent
\textit{Strong convergence of the first factor to zero in
	$L^2(0,T;L^2)$.}
By Proposition \ref{prop:uniform-Zs}, $\{v^n\}$ is bounded in
$L^2(0,T;\mathcal{H}^{5/2})$ and, as established above,
$v^n\to 0$ strongly in $L^2(0,T;\mathcal{H}^{0})$.
By complex interpolation between $\mathcal{H}^0$ and
$\mathcal{H}^{5/2}$,
\begin{equation}\label{eq:interpolation-vn}
	v^n\to 0
	\qquad\text{strongly in }
	L^2(0,T;\mathcal{H}^{\theta})
	\quad\text{for every }\theta\in[0,\tfrac52).
\end{equation}
Since $[\partial_x^{-1},G]$ is of order $-2$ in the $x$-variable, the operator
$D_x^{5/2}[\partial_x^{-1},G]$ has order at most $\frac12$. Equivalently, for every
$\rho\in\mathbb{R}$,
\[
\|D_x^{5/2}[\partial_x^{-1},G]f\|_{\mathcal{H}^{\rho-1/2}}
\le C\,\|f\|_{\mathcal{H}^{\rho}}.
\]
Choosing any $\theta\in(\frac12,\frac52)$ in \eqref{eq:interpolation-vn}, we infer that
\begin{equation}\label{eq:first-factor-to-0}
	D_x^{5/2}[\partial_x^{-1},G]\,v^n
	\to 0
	\qquad\text{strongly in }
	L^2(0,T;\mathcal{H}^{\theta-1/2})
	\subset L^2(0,T;L^2),
\end{equation}
since $\theta-\frac12>0$.

\smallskip\noindent
\textit{Boundedness of the second factor in $L^2(0,T;L^2)$.}
Since $w^n=\partial_x^{-1}v^n$ and $\{v^n\}$ is bounded in
$L^2(0,T;\mathcal{H}^{5/2})$, we have $\{w^n\}$ bounded in
$L^2(0,T;\mathcal{H}^{5/2})$ as well.  The boundedness of $G$ on
$\mathcal{H}^{5/2}$ then gives $\{Gw^n\}$ bounded in
$L^2(0,T;\mathcal{H}^{5/2})$, and therefore
\begin{equation}\label{eq:second-factor-bdd}
	\|D_x^{5/2}(G\,w^n)\|_{L^2(0,T;L^2)}
	\le C,
\end{equation}
uniformly in $n$.

\smallskip\noindent
Applying the Cauchy-Schwarz inequality to
\eqref{eq:second-summand-refactor} and using
\eqref{eq:first-factor-to-0}-\eqref{eq:second-factor-bdd},
\begin{equation}\label{eq:second-summand-to-0}
	\Bigl|\int_0^T\bigl(G\,D_x^5\,[\partial_x^{-1},G]\,v^n,\;
	w^n\bigr)\,dt\Bigr|
	\le
	\|D_x^{5/2}\,[\partial_x^{-1},G]\,v^n\|_{L^2(0,T;L^2)}\;
	\|D_x^{5/2}(G\,w^n)\|_{L^2(0,T;L^2)}
	\to 0.
\end{equation}

Combining \eqref{eq:first-summand-to-0} and
\eqref{eq:second-summand-to-0},
\begin{equation}\label{eq:Rvnwn-to-0}
	\int_0^T(Rv^n,w^n)\,dt\to 0.
\end{equation}
Inserting \eqref{eq:dissipation-vanish-X0} and
\eqref{eq:Rvnwn-to-0} into \eqref{eq:energy-identity-X0-n}, we
conclude that
\[
\|v^n(t)\|_{X_0}^2\to 1
\qquad\text{uniformly for }t\in[0,T],
\]
and therefore
\[
\|v^n\|_{L^2(0,T;X_0)}^2
=\int_0^T\|v^n(t)\|_{X_0}^2\,dt\to T>0,
\]
which contradicts the strong convergence $v^n\to v\equiv 0$ in $L^2(0,T;X_0)$, showing the Claim. Thus, the observability inequality \eqref{eq:observability-5KP-X0} is verified, and hence \eqref{eq:exp-decay-5KP-Xs} holds for $s=0$ (see item i.\
of Remark \ref{S=0}).
	
	\medskip\noindent
	\emph{\textbf{Part 2.} Extension of the decay estimate from $X_0$ to $X_s$.}
	
	\vspace{0.2cm}
	
	For $s\ge0$, consider the inhomogeneous Fourier multiplier $\Lambda^s:=\langle D_x\rangle^s$, whose symbol is $(1+k^2)^{s/2}$ for $k\neq0$. The operator $\Lambda^s$ commutes with $\partial_x^5$, $D_x^5$, and $\partial_x^{-1}\partial_y^2$, while the commutator $[\Lambda^s,G]$ is of lower order in $x$ (see Lemma \ref{lem:G-properties}). A standard commutator argument, as in Proposition \ref{prop:semigroup-5KP}, shows that $\{\mathcal{S}_\varepsilon(t)\}_{t\ge0}$ defines an analytic semigroup on $X_{s,0}$ for each $s\ge0$, with bounds on finite time intervals that are uniform with respect to $\varepsilon$.

Fix $t_0\in(0,1]$. Analyticity on $X_0$ and the sectoriality of $\mathcal{L}_\varepsilon$ yield the smoothing estimate
$$
\|\mathcal{S}_\varepsilon(t_0)f\|_{X_s}
\le C_{s,t_0}\|f\|_{X_0},
\qquad \forall f\in X_{0,0},
$$
where $C_{s,t_0}$ is independent of $\varepsilon\in(0,1)$. Using the semigroup property together with \eqref{eq:X0-exp}, for $t\ge t_0$ we obtain
\begin{align*}
\|\mathcal{S}_\varepsilon(t)v_0\|_{X_s}
&= \|\mathcal{S}_\varepsilon(t_0)\mathcal{S}_\varepsilon(t-t_0)v_0\|_{X_s}\\
&\le C_{s,t_0}\|\mathcal{S}_\varepsilon(t-t_0)v_0\|_{X_0}\\
&\le C_{s,t_0} C e^{-\lambda(t-t_0)}\|v_0\|_{X_0}\\
&\le C_{s,t_0} C e^{-\lambda(t-t_0)}\|v_0\|_{X_s},
\end{align*}
where we used the monotonicity $\|v_0\|_{X_0}\le\|v_0\|_{X_s}$ for $s\ge0$. Since $\mathcal{S}_\varepsilon(t)$ is uniformly bounded on $X_s$ for $t\in[0,t_0]$, enlarging the constant if necessary yields
\[
\|\mathcal{S}_\varepsilon(t)v_0\|_{X_s}
\le C' e^{-\lambda t}\|v_0\|_{X_s},
\qquad \forall t\ge0,
\]
with $C'$ independent of $\varepsilon$. This is precisely \eqref{eq:exp-decay-5KP-Xs}.
\end{proof}

\begin{remark}\label{S=0} We are in a position to present some remarks.
\begin{itemize}
\item[i.] Although this is not made explicit in the proof of the previous result, observe that in $X_0$ the decay estimate \eqref{eq:exp-decay-5KP-Xs} follows directly from the \emph{observability inequality} \eqref{eq:observability-5KP-X0}.  Indeed, assuming this inequality and combining it with
\eqref{eq:energy-identity-X0}, we obtain
\[
\|v(T)\|_{X_0}^2
\le \left(1-\frac{2}{C_T}\right)\|v_0\|_{X_0}^2.
\]
The semigroup property then yields exponential decay in $X_0$
\begin{equation}\label{eq:X0-exp}
\|\mathcal{S}_\varepsilon(t)v_0\|_{X_0}
\le C e^{-\lambda t}\|v_0\|_{X_0},
\qquad \forall t\ge0,
\end{equation}
for some constants $C,\lambda>0$ independent of $\varepsilon$. 
\item[ii.] We denote by $S(t)v_0$ the solution of \eqref{eq:limit-closed-loop-5KP-X} in the homogeneous case $F\equiv0$. The family $\{S(t)\}_{t\ge0}$ defines a $C_0$–semigroup on $X_{s,0}(\Omega)$ for every $s\ge0$. In particular, for $s\ge0$, the decay estimate \eqref{eq:exp-decay-5KP-Xs} can be passed to the limit as $\varepsilon\to0$ along the approximating semigroups $\{\mathcal{S}_\varepsilon(t)\}_{t\ge0}$, yielding exponential stability of the unregularized linear flow in $X_s$. This property will play a key role in the nonlinear stabilization result established later.
\end{itemize}
\end{remark}

Next, we see that the adjoint flow satisfies the same energy identity as the forward flow, with the same dissipation functional. More precisely,
\begin{proposition}\label{prop:adjoint-decay}
	Let \(0<\varepsilon<1\) and \(s\ge 0\). Then the adjoint semigroup
	\(\{\mathcal{S}_\varepsilon^*(t)\}_{t\ge0}\), generated by
	\(-\mathcal{L}_\varepsilon^*\) on \(X_{0,0}(\Omega)\), satisfies the same
	uniform exponential decay estimate as \(\mathcal{S}_\varepsilon(t)\). More precisely,
	there exist constants \(C,\lambda>0\), independent of \(\varepsilon\), such that
	\begin{equation}\label{eq:adjoint-decay}
		\|\mathcal{S}_\varepsilon^*(t)v_0\|_{X_s}
		\le Ce^{-\lambda t}\|v_0\|_{X_s},
		\qquad \forall\, t\ge0,\quad v_0\in X_{s,0}(\Omega).
	\end{equation}
\end{proposition}

\begin{proof}
	We begin by observing that \(\mathcal{L}_\varepsilon^*\) preserves the \(x\)-mean-zero subspace \(X_{0,0}(\Omega)\). Moreover,
	\[
	\mathcal{L}_\varepsilon^*
	=\varepsilon D_x^5+GD_x^5G+\beta\,\partial_x^5-\gamma\,\partial_x^{-1}\partial_y^2,
	\]
	so that the dissipative part of the adjoint generator coincides with that of \(\mathcal{L}_\varepsilon\), whereas the skew-adjoint part appears only with the opposite sign in the transverse term. In particular, for every
	\(f\in X_{5,0}(\Omega)\),
	\[
	\mathrm{Re}\,(\mathcal{L}_\varepsilon^*f,f)_{L^2}
	=\varepsilon\|D_x^{5/2}f\|_{L^2}^2
	+\|D_x^{5/2}(Gf)\|_{L^2}^2.
	\]
	It follows that the smoothing estimate obtained from the localized dissipation, together with the propagation-of-regularity argument developed earlier for \(\mathcal{S}_\varepsilon(t)\), remains valid for \(\mathcal{S}_\varepsilon^*(t)\) with constants independent of \(\varepsilon\). In particular, the adjoint trajectories enjoy the same uniform bounds in the corresponding \(Z^{s,T}\)-type spaces.
	
	We now claim that the observability argument leading to Proposition \ref{prop:exp-decay-5KP} carries over to the adjoint flow. Indeed, the contradiction procedure only uses the uniform energy estimate, the propagated regularity, the compactness properties of the associated trajectories, and the unique continuation statement for the limiting equation. For the adjoint dynamics, the limit equation takes the form
	\[
	\partial_t v+\beta\,\partial_x^5 v-\gamma\,\partial_x^{-1}\partial_y^2 v
	+\varepsilon_*D_x^5 v=0,
	\]
	for some \(\varepsilon_*\in[0,1]\). This equation differs from the forward one only by the sign of the skew-adjoint component \(\partial_x^{-1}\partial_y^2\). Since Theorem \ref{Th:UCP-5KP} is formulated without any restriction on the signs of \(\beta\) and \(\gamma\), the same unique continuation property applies to this limit equation as well. Therefore, the same contradiction argument yields a uniform observability inequality for the adjoint semigroup, and hence the same exponential stability estimate follows. This proves \eqref{eq:adjoint-decay}.
\end{proof}

The last result of this section gives the decay properties when $\varepsilon=0$.
\begin{proposition}\label{prop:limit-semigroup-decay}
	Let \(s\ge 0\). Then the \(C_0\)-semigroup \(\{\mathcal{S}(t)\}_{t\ge0}\) on
	\(X_{s,0}(\Omega)\), constructed in Proposition \ref{prop:BS-limit-5KP-X},
	satisfies the exponential decay estimate
	\begin{equation}\label{eq:limit-decay-explicit}
		\|\mathcal{S}(t)v_0\|_{X_s}
		\le Ce^{-\lambda t}\|v_0\|_{X_s},
		\qquad \forall\, t\ge0,\quad v_0\in X_{s,0}(\Omega),
	\end{equation}
	where \(C,\lambda>0\) are the same constants as in
	Proposition \ref{prop:exp-decay-5KP}.
\end{proposition}

\begin{proof}
	Fix \(v_0\in X_{s,0}(\Omega)\). Let \(\{v_0^n\}\subset X_{5,0}(\Omega)\) and
	\(\{\varepsilon_n\}\subset(0,1)\) be the approximating sequences introduced in the
	Bona-Smith construction of Proposition \ref{prop:BS-limit-5KP-X}, so that
	\[
	v_0^n\to v_0 \quad \text{in } X_s,
	\qquad
	\varepsilon_n\to 0.
	\]
	For each \(n\), define
	\[
	v^n(t):=\mathcal{S}_{\varepsilon_n}(t)v_0^n.
	\]
	Since the exponential decay estimate \eqref{eq:exp-decay-5KP-Xs} is uniform with
	respect to \(\varepsilon\in(0,1)\), we have
	\[
	\|v^n(t)\|_{X_s}
	\le Ce^{-\lambda t}\|v_0^n\|_{X_s},
	\qquad \forall\, t\ge0,
	\]
	with constants \(C,\lambda>0\) independent of \(n\).
	
	On the other hand, by the strong convergence statement
	\eqref{eq:strong-conv-semigroup-X}, one has
	\[
	v^n(t)\to \mathcal{S}(t)v_0
	\quad \text{in } X_s,
	\]
	locally uniformly for \(t\) in bounded intervals. Therefore, for each fixed
	\(t\ge0\),
	\[
	\|\mathcal{S}(t)v_0\|_{X_s}
	=\lim_{n\to\infty}\|v^n(t)\|_{X_s}
	\le Ce^{-\lambda t}\lim_{n\to\infty}\|v_0^n\|_{X_s}
	=Ce^{-\lambda t}\|v_0\|_{X_s}.
	\]
	This proves \eqref{eq:limit-decay-explicit}.
\end{proof}

\subsection{Nonlinear stabilization: Proof of Theorem \ref{thm:stab5KP}}The global well-posedness statement and the contraction framework have
already been established in Theorem \ref{thm:gwp-kp5-X}. It remains only
to verify the exponential decay estimate \eqref{eq:nonlin-exp-decay-kp5}.

Let $u\in\mathcal{X}$ be the unique global solution given by
Theorem \ref{thm:gwp-kp5-X}, with $\|u_0\|_{X_s}\le\rho$ and
$\|u\|_{\mathcal{X}}\le R$. By the definition of the $\mathcal{X}$-norm,
\[
\vertiii{u}_n \le e^{-\lambda n}\|u\|_{\mathcal{X}}
\le e^{-\lambda n}R,
\qquad \forall\,n\ge0.
\]
For $t\in[n,n+1]$, one obtains
\[
\|u(t)\|_{X_s}
\le \vertiii{u}_n
\le R\,e^{-\lambda n}
\le R\,e^{\lambda}\,e^{-\lambda t}.
\]
Since $R$ depends only on $s$ and the universal constants $\tilde c_0,\tilde c_1,C_\lambda$, while $\|u_0\|_{X_s}\le\rho=R/(2\tilde c_0)$,
we conclude that
\[
\|u(t)\|_{X_s}
\le C\,e^{-\lambda t}\|u_0\|_{X_s}
\qquad \forall\,t\ge0,
\]
with $C:=2\tilde c_0 e^{\lambda}$, which is precisely
\eqref{eq:nonlin-exp-decay-kp5}. \qed

\subsection{Exact controllability: Linear result}

Local exact controllability for \eqref{control5KP_q} follows from the Hilbert Uniqueness Method (HUM) of \cite{Lions1988}, combined with smoothing estimates in $Z^{s,T}$ and a fixed-point argument in $X_{s,0}(\Omega)$. We consider
\begin{equation}\label{eq:linear-control}
\left\{
\begin{aligned}
\partial_t v
+\beta\,\partial_x^5 v
+\gamma\,\partial_x^{-1}\partial_y^2 v
+G D_x^5 G v
&= G D_x^{5/2} h,\\
v(0) &= v_0,
\end{aligned}
\right.
\end{equation}
with $v_0\in X_{s,0}$ and $h\in L^2(0,T;X_{s,0})$.

On the \(x\)-mean-zero subspace \(X_{0,0}(\Omega)\), both \(G\) and \(D_x^{5/2}\) are self-adjoint on \(L^2(\Omega)\); see
Lemma \ref{lem:G-properties}\,(ii) and the definition of the Fourier multiplier \(D_x^r\). In general, however, these operators do not commute. Therefore, if one defines the control operator by
	\[
	B:=GD_x^{5/2},
	\]
	then its \(L^2(\Omega)\)-adjoint is
	\[
	B^*=(GD_x^{5/2})^*=D_x^{5/2}G,
	\]
	and in particular \(B^*\neq B\) in general.
	
	Accordingly, the controlled equation \eqref{eq:linear-control} is written with input operator \(B=GD_x^{5/2}\), whereas the observation term arising in the duality identity for the adjoint system is \(B^*=D_x^{5/2}G\). The observability inequality
\eqref{eq:observability-5KP-X0} is thus naturally expressed in terms of
	\[
	\|D_x^{5/2}(Gw)\|_{L^2(\Omega)}^2=\|B^*w\|_{L^2(\Omega)}^2,
	\]
	which is the quantity required in the HUM construction. With this notation, the control term in \eqref{control5KP_q} may be written as
	\[
	(D_x^{5/2}G)^*q=(B^*)^*q=Bq=GD_x^{5/2}q,
	\]
	which is consistent with the form of the controlled equation.

Note that the adjoint system reads
\begin{equation}\label{eq:adjoint}
\left\{
\begin{aligned}
-\partial_t w
+\beta\,\partial_x^5 w
+\gamma\,\partial_x^{-1}\partial_y^2 w
+G D_x^5 G w &= 0,\\
w(T) &= w_T.
\end{aligned}
\right.
\end{equation}
By time reversal, this coincides with the closed-loop system, hence
$$
w\in Z^{s,T}, \qquad
\|w\|_{Z^{s,T}}\le C\|w_T\|_{X_s}.
$$
A standard integration by parts yields
$$
(v(T),w_T) - (v_0,w(0))
= \int_0^T (h, D_x^{5/2} G w)\,dt.
$$
So, define the HUM operator
$$
\Lambda w_T := \Phi(w_T)(T),
$$
where $\Phi(w_T)$ solves \eqref{eq:linear-control} with $v_0=0$ and $h=D_x^{5/2}Gw$. Then
$$
(\Lambda w_T,w_T)
= \int_0^T \|D_x^{5/2}(Gw)\|_{L^2}^2\,dt.
$$
Using the observability inequality,
$$
\|w_T\|_{L^2}^2
\le C \int_0^T \|D_x^{5/2}(Gw)\|_{L^2}^2\,dt,
$$
we deduce that $\Lambda$ is bounded, symmetric, and coercive on $L^2$, hence an isomorphism (and also on $X_{s,0}$ by interpolation). 

We may now complete the linear HUM construction and formulate the exact controllability result in the natural KP-adapted Sobolev scale.
\begin{proposition}
	Let \(T>0\) and \(s\ge 0\). For any
	\(v_0,v_1\in X_{s,0}(\Omega)\), there exists a control
	\(h\in L^2(0,T;X_{s,0}(\Omega))\) such that the solution of
	\eqref{eq:linear-control} satisfies \(v(T)=v_1\). Moreover,
	\begin{equation}\label{eq:linear-control-cost}
		\|h\|_{L^2(0,T;X_s)}
		\le C(s,T)\bigl(\|v_0\|_{X_s}+\|v_1\|_{X_s}\bigr).
	\end{equation}
\end{proposition}

\begin{proof}
	Let
	\[
	g:=v_1-\mathcal{S}(T)v_0\in X_{s,0}(\Omega),
	\]
	and define
	\[
	w_T:=\Lambda^{-1}g\in X_{s,0}(\Omega).
	\]
	Let \(w\) be the solution of the adjoint system with terminal condition
	\(w(T)=w_T\), and set
	\[
	h:=D_x^{5/2}Gw.
	\]
	By the boundedness of \(\Lambda^{-1}\) on \(X_{s,0}\) and
	Proposition  \ref{prop:BS-limit-5KP-X}, one has
	\[
	\|h\|_{L^2(0,T;X_s)}
	\le C\|w\|_{Z^{s,T}}
	\le C\|w_T\|_{X_s}
	\le C\|g\|_{X_s}.
	\]
	Since \(\mathcal{S}(T)\) is bounded on \(X_{s,0}(\Omega)\), it follows that
	\[
	\|g\|_{X_s}\le C(T)\bigl(\|v_0\|_{X_s}+\|v_1\|_{X_s}\bigr),
	\]
	which yields \eqref{eq:linear-control-cost}. The identity \(v(T)=v_1\)
	follows from the definition of \(\Lambda\).
\end{proof}

\subsection{Nonlinear controllability: Proof of Theorem \ref{thm:control5KP}}
Consider the control system
\begin{equation}\label{eq:nonlin-control}
\left\{
\begin{aligned}
\partial_t u
+\beta\,\partial_x^5 u
+\gamma\,\partial_x^{-1}\partial_y^2 u
+G D_x^5 G u
&= -u\,\partial_x u + G D_x^{5/2} q,\\
u(0)=u_0,\quad u(T)=u_1.
\end{aligned}
\right.
\end{equation}
For each \(w_T\in X_{s,0}(\Omega)\), let \(w\) denote the solution of the adjoint system \eqref{eq:adjoint} with terminal condition \(w(T)=w_T\), and define
	\[
	h_{w_T}:=D_x^{5/2}Gw \in L^2(0,T;X_{s,0}).
	\]
	Let \(\Phi(w_T)\in Z^{s,T}\) be the corresponding solution of the linear controlled system \eqref{eq:linear-control} with initial datum \(v_0=0\) and control \(h=h_{w_T}\).  The associated HUM operator is then given by
	\[
	\Lambda w_T:=\Phi(w_T)(T)\in X_{s,0}.
	\]
	
	By Proposition \ref{prop:BS-limit-5KP-X}, the map
	\[
	w_T\longmapsto \Phi(w_T)
	\]
	is continuous from \(X_{s,0}\) into \(Z^{s,T}\). Moreover, the observability inequality \eqref{eq:observability-5KP-X0} implies that \(\Lambda\) is bounded, symmetric, and coercive on \(L^2(\Omega)\). Indeed,
$$
		(\Lambda w_T,w_T)_{L^2}
		=\int_0^T \|D_x^{5/2}(Gw)\|_{L^2}^2\,dt
		\ge C_T^{-1}\|w_T\|_{L^2}^2 .
$$
	Hence, by the Lax-Milgram theorem, \(\Lambda\) is an isomorphism on \(L^2(\Omega)\).
	
	We next show that \(\Lambda\) acts continuously on \(X_{s,0}\) and that its inverse
	preserves the same regularity. If \(w_T\in X_{s,0}\), then Proposition  \ref{prop:BS-limit-5KP-X}
	yields \(w\in Z^{s,T}\), so that \(h_{w_T}=D_x^{5/2}Gw\in L^2(0,T;X_{s,0})\), and therefore
	\(\Phi(w_T)\in Z^{s,T}\). In particular,
	\[
	\|\Lambda w_T\|_{X_s}
	=\|\Phi(w_T)(T)\|_{X_s}
	\le C(s,T)\|w_T\|_{X_s},
	\]
	so that \(\Lambda:X_{s,0}\to X_{s,0}\) is bounded. Let now \(g\in X_{s,0}\subset L^2(\Omega)\).
	Since \(\Lambda\) is invertible on \(L^2(\Omega)\), there exists \(w_T\in L^2(\Omega)\) such that
	\(\Lambda w_T=g\). Invoking the adjoint observability estimate in \(X_{-s}\)-\(X_s\) duality,
	as in \cite[Prop. 4.1]{FloreSmith2019} and \cite[Sec. 4]{LR}, one obtains
$$
		\|w_T\|_{X_s}\le C(s,T)\|\Lambda w_T\|_{X_s}
		= C(s,T)\|g\|_{X_s}.
$$
	It follows that \(w_T\in X_{s,0}\), and consequently \(\Lambda^{-1}:X_{s,0}\to X_{s,0}\) is bounded:
	\begin{equation}\label{eq:Lambda-inv-bound-rev}
		\|\Lambda^{-1}g\|_{X_s}\le C(s,T)\|g\|_{X_s}.
	\end{equation}
	
	We now turn to the nonlinear problem. For \(u\in Z^{s,T}\), set $f_u:=-u\,\partial_x u$, and define
	\[
	g_u:=u_1-\mathcal{S}(T)u_0
	+\int_0^T \mathcal{S}(T-\tau)f_u(\tau)\,d\tau.
	\]
	Let \(w_u\) be the solution of the adjoint system \eqref{eq:adjoint} with terminal datum
	\[
	w_T=\Lambda^{-1}g_u,
	\]
	and define the corresponding control
	\[
	q_u:=D_x^{5/2}G\,w_u.
	\]
	We then let \(\Gamma_c(u)\) denote the solution of
	\begin{equation}\label{eq:Gamma-c-def-rev}
		\left\{
		\begin{aligned}
			\partial_t v
			+\beta\,\partial_x^5 v
			+\gamma\,\partial_x^{-1}\partial_y^2 v
			+GD_x^5Gv
			&= f_u + GD_x^{5/2}q_u,\\
			v(0)&=u_0.
		\end{aligned}
		\right.
	\end{equation}
By construction, \(\Gamma_c(u)\) satisfies the terminal condition \(\Gamma_c(u)(T)=u_1\), for every \(u\in Z^{s,T}\). Indeed, by Duhamel's formula, the definition of \(g_u\), and the identity \(\Lambda w_T=\Phi(w_T)(T)\), one finds
	\[
	\Gamma_c(u)(T)
	=\mathcal{S}(T)u_0
	-\int_0^T \mathcal{S}(T-\tau)f_u(\tau)\,d\tau
	+\Lambda(\Lambda^{-1}g_u)
	=u_1.
	\]
	
	We claim that \(\Gamma_c\) is a contraction on a sufficiently small ball in \(Z^{s,T}\).
	Fix \(R>0\) and consider
	\[
	B_R:=\{u\in Z^{s,T}:\ \|u\|_{Z^{s,T}}\le R\}.
	\]
	Applying the linear estimate \eqref{eq:apriori-limit-5KP-X} to \eqref{eq:Gamma-c-def-rev},
	we obtain
	\[
	\|\Gamma_c(u)\|_{Z^{s,T}}
	\le C(s,T)\Bigl(\|u_0\|_{X_s}
	+\|f_u\|_{L^2(0,T;X_{s-5/2})}
	+\|q_u\|_{L^2(0,T;X_s)}\Bigr).
	\]
	The bilinear estimate from Lemma \ref{lem:bilinear-kp5-X} yields
	\begin{equation}\label{eq:fu-bound-rev}
		\|f_u\|_{L^2(0,T;X_{s-5/2})}
		\le C\|u\|_{Z^{s,T}}^2.
	\end{equation}
	On the other hand, by Proposition \ref{prop:BS-limit-5KP-X},
	\eqref{eq:Lambda-inv-bound-rev}, and the definition of \(g_u\),
	\[
	\|q_u\|_{L^2(0,T;X_s)}
	\le C\|w_u\|_{Z^{s,T}}
	\le C\|w_T\|_{X_s}
	\le C\|\Lambda^{-1}g_u\|_{X_s}
	\le C\|g_u\|_{X_s},
	\]
	and therefore, using again \eqref{eq:fu-bound-rev},
	\begin{equation}\label{eq:qu-bound-rev}
		\|q_u\|_{L^2(0,T;X_s)}
		\le C(s,T)\Bigl(\|u_0\|_{X_s}+\|u_1\|_{X_s}
		+\|u\|_{Z^{s,T}}^2\Bigr).
	\end{equation}
	Combining the preceding estimates, we arrive at
	\[
	\|\Gamma_c(u)\|_{Z^{s,T}}
	\le C_0\bigl(\|u_0\|_{X_s}+\|u_1\|_{X_s}\bigr)
	+C_1\|u\|_{Z^{s,T}}^2,
	\]
	for some constants \(C_0,C_1>0\) depending only on \(s\) and \(T\).
	
	A similar argument applied to the difference of two trajectories shows that, for
	\(u,v\in B_R\),
	\[
	\|\Gamma_c(u)-\Gamma_c(v)\|_{Z^{s,T}}
	\le C_1\bigl(\|u\|_{Z^{s,T}}+\|v\|_{Z^{s,T}}\bigr)\|u-v\|_{Z^{s,T}}
	\le 2C_1R\,\|u-v\|_{Z^{s,T}}.
	\]
	Choose
	\[
	R:=\frac{1}{4C_1},
	\qquad
	\delta:=\frac{R}{2C_0}.
	\]
	If
	\[
	\|u_0\|_{X_s}+\|u_1\|_{X_s}\le \delta,
	\]
	then \(\Gamma_c\) maps \(B_R\) into itself and is a strict contraction on \(B_R\). The Banach fixed-point theorem then yields a unique \(u\in B_R\subset Z^{s,T}\) such that \(\Gamma_c(u)=u\). Defining \(q:=q_u\), we conclude that the pair \((u,q)\) solves \eqref{eq:nonlin-control} with $u(0)=u_0,$ and $u(T)=u_1$. 
	
	Finally, the control cost estimate follows from \eqref{eq:qu-bound-rev} together with the bound \(\|u\|_{Z^{s,T}}\le R\):
	\[
	\|q\|_{L^2(0,T;X_s)}
	\le C(s,T)\bigl(\|u_0\|_{X_s}+\|u_1\|_{X_s}\bigr).
	\]
	This completes the proof. \qed

\section{Concluding remarks and perspectives}\label{Secfinal}

\subsection{Extensions to general nonlinearities} The analysis developed in this work is robust and extends to a broad class of fifth-order KP-type equations, in close analogy with higher-order KdV models \cite{Grunrock2010,KenigPilod2016,FloreSmith2019}. The key ingredients, the feedback structure $G D_x^5 G$, the KP-adapted scale $X_s$, and the smoothing properties of the linear flow, allow one to treat nonlinearities that are polynomials in $u$ and its $x$-derivatives up to order three.

More precisely, consider
$$
\partial_t u
+ \beta\,\partial_x^5 u
+ \gamma\,\partial_x^{-1}\partial_y^2 u
+ \mathcal{N}\bigl(u,\partial_x u,\partial_x^2 u,\partial_x^3 u\bigr), 
+ G D_x^5 G u = 0,\qquad \text{on }\Omega=\mathbb{T}\times\mathbb{R},
$$
where $\mathcal{N}$ satisfies:
\begin{itemize}
\item $\mathcal{N}(0)=0$ (equilibrium at the origin);
\item preservation of the $x$-mean;
\item local Lipschitz continuity
$\mathcal{N}:X_{s,0}\to X_{s-5/2,0}$ for $s>2$.
\end{itemize}
These assumptions are natural in the $X_s$-framework and are ensured by the bilinear estimates established in this paper. Typical examples for $\mathcal{N}\bigl(u,\partial_x u,\partial_x^2 u,\partial_x^3 u\bigr)$ include $u^2\partial_x u$, $\partial_x u\,\partial_x^2 u + u\,\partial_x^3 u$, and mixed polynomial nonlinearities. Under these conditions, the main results of the paper extend with only minor modifications:
\begin{itemize}
\item local well-posedness follows from the same contraction argument in $Z^{s,T}$;
\item exponential stabilization holds in the small-data regime by combining linear decay with nonlinear estimates;
\item local exact controllability is obtained via the same HUM-based fixed-point scheme, since the nonlinearity enters only through estimates in $L^2(0,T;X_{s-5/2})$.
\end{itemize}

In summary, the feedback law $G D_x^5 G$ provides a unified mechanism ensuring well-posedness, stabilization, and controllability for a wide class of fifth-order KP-type models.

\subsection{Perspectives on the bidimensional torus} A natural extension of this work is the study of the fifth-order KP equation on
the fully periodic domain $\mathbb{T}^2=\mathbb{T}_x\times\mathbb{T}_y$:
$$
\partial_t u + \beta\,\partial_x^5 u + u\,\partial_x u
+ \gamma\,\partial_x^{-1}\partial_y^2 u = 0.
$$
While this setting is conceptually close, it introduces several substantial difficulties.

\medskip
\noindent
\textbf{(i) Mean-zero structure.} A suitable mean-zero constraint is required to define $\partial_x^{-1}$ as a Fourier multiplier. On $\mathbb{T}^2$, this is typically imposed globally, leading to a reduced phase space analogous to $L^2_0(\mathbb{T}^2)$, in the same spirit as in \cite{RivasSun2020} for the
third-order KP-II equation.

\medskip
\noindent
\textbf{(ii) Fully discrete spectrum.} Unlike the cylindrical geometry $\mathbb{T}\times\mathbb{R}$, the torus $\mathbb{T}^2$ yields a discrete spectrum in both variables. The symbol $\beta k^5 - \gamma \ell^2/k$ exhibits a more intricate structure, and the analysis of propagation and dispersion requires new spectral or microlocal tools.

\medskip
\noindent
\textbf{(iii) Well-posedness issues.} Although significant progress has been made for fifth-order KP equations \cite{ST1,ST2,Robert2019,KleinSaut2021}, a complete well-posedness theory in the KP-adapted scale $X_{s,0}(\mathbb{T}^2)$ remains open. Bridging the gap between energy-based approaches and the functional framework used here is a key step toward control applications.

\medskip
\noindent
\textbf{(iv) Unique continuation and observability.} The unique continuation property proved in this work relies strongly on the mixed geometry $\mathbb{T}\times\mathbb{R}$, combining analyticity in $x$ and Fourier reduction in $y$. In the fully periodic case, the absence of continuous frequencies and the discrete coupling between modes suggest that new microlocal or semiclassical techniques will be necessary. Whether the geometric features observed for third-order KP equations, such as the
distinction between vertical and horizontal control regions, persist at fifth order remains an open question.
	
	\subsection*{Acknowledgment} This research was conducted during a series of academic visits between the authors at the Federal University of Pernambuco and the Federal University of Piauí. The authors sincerely appreciate the hospitality and institutional support extended by both universities.

\subsection*{Funding} Capistrano-Filho was partially supported by CNPq grant numbers 301744/2025-4 and  421573/2023-6, CAPES/COFECUB grant number 88887.879175/2023-00. Capistrano-Filho and Nascimento was partially supported by PROPG (UFPE) \textit{via} PROAP resources.


\begin{thebibliography}{99}

\bibitem{AmannLinear1995}
H. Amann,
\emph{Linear and Quasilinear Parabolic Problems, Vol. I: Abstract Linear Theory},
Monographs in Mathematics, vol.~89, Birkh\"auser, Basel, 1995.

\bibitem{Aubin}
J.-P.~Aubin,
\emph{Un théorème de compacité},
 C. R. Acad. Sci. Paris, \textbf{256} (1963), 5042--5044. 

\bibitem{BoSm1975}
J.~L.~Bona and R.~Smith,
\emph{The initial-value problem for the Korteweg--de Vries equation},
Philos.\ Trans.\ Roy.\ Soc.\ London Ser.\ A \textbf{278} (1975), 555--601.

\bibitem{Boukarou2021}
A.~Boukarou, K.~Zennir, and S.~Georgiev,
\emph{Global well-posedness for the fifth-order Kadomtsev--Petviashvili II equation in anisotropic Gevrey spaces},
Discrete Contin.\ Dyn.\ Syst.\ \textbf{18} (2021), no.~2, 119--136.

 \bibitem{Corduneanu1989}
  C.~Corduneanu,
  \emph{Almost Periodic Functions},
  2nd ed., Chelsea, New York, 1989.



\bibitem{Craig2005}
W.~Craig, P.~Guyenne, and H.~Kalisch,
\emph{Hamiltonian long-wave expansions for free surfaces and interfaces},
Comm.\ Pure Appl.\ Math.\ \textbf{58} (2005), 1587--1641.

\bibitem{EngelNagel2000}
K.-J.~Engel and R.~Nagel,
\emph{One-Parameter Semigroups for Linear Evolution Equations},
Graduate Texts in Mathematics, vol.~194, Springer, New York, 2000.

\bibitem{Erbay2022}
H.~A.~Erbay, S.~Erbay, and A.~Erkip,
\emph{On the full dispersion Kadomtsev--Petviashvili equations for internal waves over variable topography},
Stud.\ Appl.\ Math.\ \textbf{149} (2022), 702--747.

\bibitem{FloreSmith2019}
C.~Flores and D.~L.~Smith,
\emph{Control and stabilization of the periodic fifth-order Korteweg--de Vries equation},
ESAIM Control Optim.\ Calc.\ Var.\ \textbf{25} (2019), Art.~38.


\bibitem{Grunrock2010}
A.~Gr\"unrock,
\emph{On the hierarchies of higher order mKdV and KdV equations},
Cent.\ Eur.\ J.\ Math.\ \textbf{8} (2010), 500--536.



\bibitem{IoKe2007}
A.~D.~Ionescu and C.~E.~Kenig,
\emph{Global well-posedness of the Benjamin--Ono equation in low-regularity spaces},
J.\ Amer.\ Math.\ Soc.\ \textbf{20} (2007), 753--798.


\bibitem{KenigPilod2016}
C.~E.~Kenig and D.~Pilod,
\emph{Local well-posedness for the KdV hierarchy at high regularity},
Adv.\ Differential Equations \textbf{21} (2016), 801--836.

\bibitem{KleinSaut2021}
C.~Klein and J.-C.~Saut,
\emph{Nonlinear dispersive equations---inverse scattering and PDE methods},
Appl.\ Math.\ Sci.\ \textbf{209}, Springer, Cham, 2021.

\bibitem{KP1970}
B.~B.~Kadomtsev and V.~I.~Petviashvili,
\emph{On the stability of solitary waves in weakly dispersive media},
Sov.\ Phys.\ Dokl.\ \textbf{15} (1970), 539--541.




\bibitem{LiXiao2008}
J.~Li and J.~Xiao,
\emph{Well-posedness of the fifth order Kadomtsev--Petviashvili I equation in anisotropic Sobolev spaces},
J.\ Math.\ Pures Appl.\ \textbf{90} (2008), 338--352.

\bibitem{LiYanZhang2017}
Y.~Li, W.~Yan, and Y.~Zhang,
\emph{Global well-posedness of a fifth-order KP-I equation},
preprint arXiv:1712.09334 (2017).

\bibitem{Lions1988}
J.-L.~Lions,
\emph{Contr\^olabilit\'e exacte et stabilisation de syst\`emes distribu\'es, Vol.~I},
Masson, Paris, 1988.

\bibitem{LR}
F.~Linares and L.~Rosier,
\emph{Control and stabilization of the Benjamin--Ono equation on a periodic domain},
Trans.\ Amer.\ Math.\ Soc.\ \textbf{367} (2015), 4595--4626.

\bibitem{Molinet2011}
L.~Molinet, J.-C.~Saut, and N.~Tzvetkov,
\emph{Global well-posedness for the KP-II equation on the background of a non-localized solution},
Ann.\ Inst.\ H.\ Poincar\'e Anal.\ Non Lin\'eaire \textbf{28} (2011), 653--676.


\bibitem{RivasSun2020}
I.~Rivas and C.~Sun,
\emph{Internal controllability of nonlocalized solutions for the KP-II equation},
SIAM J.\ Control Optim.\ \textbf{58} (2020), 1715--1734.

\bibitem{Robert2019}
T.~Robert,
\emph{On the Cauchy problem for the periodic fifth-order KP-I equation},
Differential Integral Equations \textbf{32} (2019), 679--704.



\bibitem{Simon} 
J.~Simon, 
\emph{Compact sets in the space $L^p(0,T;B)$,} 
Annali di Matematica Pura ed Applicata. \textbf{146} (1986), 65--96.

\bibitem{ST1}
J.-C.~Saut and N.~Tzvetkov,
\emph{The Cauchy problem for higher-order KP equations},
J.\ Differential Equations \textbf{153} (1999), 196--222.

\bibitem{ST2}
J.-C.~Saut and N.~Tzvetkov,
\emph{The Cauchy problem for the fifth-order KP equations},
J.\ Math.\ Pures Appl.\ \textbf{79} (2000), 307--338.

\bibitem{Tzvetkov2008}
N.~Tzvetkov,
\emph{Transverse stability issues in Hamiltonian PDE},
in \emph{European Congress of Mathematics}, EMS, 2008, pp.~443--458.

\end{thebibliography}
\end{document}